\documentclass[final]{siamltex}

\expandafter\let\csname equation*\endcsname\relax
\expandafter\let\csname endequation*\endcsname\relax
\usepackage{graphicx,amssymb,amsmath}
\usepackage{epsfig}
\usepackage{subfigure}
\usepackage{cite}
\usepackage{booktabs}
\usepackage{array}
\usepackage{color}
\usepackage{enumerate}
\usepackage{algorithm}
\usepackage{algorithmic}
\usepackage{float}
\usepackage{multirow}
\usepackage{rotating}
\usepackage[misc]{ifsym}
\usepackage{bm}
\usepackage{mathrsfs}
\usepackage[pdfstartview=FitH,colorlinks,pdfborder=001,linkcolor=blue,anchorcolor=red,citecolor=blue]{hyperref}
\usepackage[top=1.3in, bottom=1.0in, left=1.3in, right=1.3in]{geometry}
\newtheorem{example}{Example}[section]

\title{An ADMM-LAP method for total variation myopic deconvolution of adaptive optics retinal images}

\author{Xiaotong Chen\thanks{School of Mathematical Sciences, Dalian University of Technology, Dalian, Liaoning 116025, China. ({\tt chenxiaotong@mail.dlut.edu.cn, yubo@dlut.edu.cn}).}
\and James~L.~Herring\thanks{Department of Mathematics, University of Houston, Houston, TX 77204, U.S.A. ({\tt herrinj@gmail.com}).}
\and James~G.~Nagy\thanks{Department of Mathematics, Emory University, Atlanta, GA 30322, U.S.A. ({\tt jnagy@emory.edu, yxi26@emory.edu}).}
\and Yuanzhe~Xi\footnotemark[3]
\and Bo Yu\footnotemark[1]
}

\begin{document}

\maketitle

\begin{abstract}
Adaptive optics (AO) corrected flood imaging of the retina is a popular technique for studying the retinal structure and function in the living eye. However, the raw retinal images are usually of poor contrast and the interpretation of such images requires image deconvolution. Different from standard deconvolution problems where the point spread function (PSF) is completely known, the PSF in these retinal imaging problems is only partially known which leads to the more complicated myopic (mildly blind) deconvolution problem. In this paper, we propose an efficient numerical scheme for solving this myopic deconvolution problem with total variational (TV) regularization. {First, we apply the alternating direction method of multipliers (ADMM) to tackle the TV regularizer. Specifically, we reformulate the TV problem as an equivalent equality constrained problem where the objective function is separable, and then minimize the augmented Lagrangian function by alternating between two (separated) blocks of unknowns to obtain the solution.} Due to the structure of the retinal images, the subproblems with respect to the fidelity term appearing within each ADMM iteration are tightly coupled and a variation of the Linearize And Project (LAP) method is designed to solve these subproblems efficiently. The proposed method is called the ADMM-LAP method. Theoretically, we establish the subsequence convergence of the ADMM-LAP method to a stationary point. Both the theoretical complexity analysis and numerical results are provided to demonstrate the efficiency of the ADMM-LAP method.
\end{abstract}



\section{Introduction}
\label{intro}
Image restoration is an important topic in image processing which is widely used in many areas, such as astronomical imaging, medical imaging and restoring aging and deteriorated films. The goal of image restoration is to reconstruct the best possible approximation of the clean, original image from an observed, blurred and noisy image. The basic image restoration problem can be described as a linear inverse problem
\begin{equation}
\bm d=A\bm x+\bm e,
\label{eq:pb1}
\end{equation}
where $A\in \mathbb{R}^{n^{2}\times n^{2}}$ is an ill-conditioned blurring matrix defined by the point spread function (PSF)\cite{Hansen2006}, $\bm d\in \mathbb{R}^{n^2}$ represents the observed, blurred and noisy image, $\bm e\in \mathbb{R}^{n^2}$ denotes the additive noise and $\bm x\in \mathbb{R}^{n^2}$ denotes the unknown true image to be restored.

The PSF is generally assumed to be perfectly known in standard image restoration techniques. However, this is not always the case. In many applications, the true PSF (and therefore the blurring matrix $A$) is unknown or partially known. This results in blind deconvolution problems where image restoration also requires recovering or approximating the PSF. If the PSF can be parameterized by a small number of unknown parameters, the problem
can be considered mildly/partially blind, or {\em myopic}.
Adaptive optics (AO) corrected flood imaging of the retina is one such myopic deconvolution problem that has received much attention. In \cite{Blanco2011, Blanco2014}, the authors present an imaging model  that can transfer the 3D model to a 2D model with the global PSF being an unknown linear combination of a few PSFs. Thus, problem \eqref{eq:pb1} requires estimating both the combination coefficients of $A$ and the true image $\bm x$ in this model.

In order to guarantee the fidelity of the recovery, it is necessary to add a regularizer. There are two well-known types of regularizers for problem \eqref{eq:pb1}: one is Tikhonov and the other is total variation (TV). Tikhonov regularization was first proposed in \cite{Tikhonov1977} with a quadratic penalty added to the objective function. Due to its quadratic property, it is inexpensive to minimize the objective function. Thus, Tikhonov regularization is computationally efficient and has been widely used. However, one disadvantage of the Tikhonov approach is that it tends to over smooth the image and fails to preserve important image details such as sharp edges. The TV regularizer, first proposed by Rudin, Osher and Fatemi in \cite{Osher1992}, has also been widely adopted in image reconstruction problems \cite{Osher1994,Wang2008, Tao2009, Hu2019}. Since the TV approach uses the summation of the variation of the image $\bm x$ at all pixels to control the norm (or semi-norm) and the smoothness of the solution, it has been shown both experimentally and theoretically that the TV approach can effectively preserve sharp edges and keep the important features of the restored image.

Following the AO retinal imaging model in \cite{Blanco2011, Blanco2014}, in this paper, we consider the myopic deconvolution model with TV regularization as follows
\begin{equation}
\label{eq:objFctn1}
\begin{aligned}
\min_{\bm x \in \mathcal{C}_{\bm x}, \bm w \in \mathcal{C}_{\bm w}}  &\Phi\left(\bm x, \bm w\right) = \frac{\mu}{2} \|{ A}\left(\bm w\right) \bm x - \bm d \|_2^2+{\rm TV}\left(\bm x\right) \\
\text{s.t.} \quad \  &\sum_{j = 1}^p w_j= 1,\\
\end{aligned}
\end{equation}
where $\mathcal{C}_{\bm x} = \left\{ \bm x \mid x_i \geq 0 \text{ for } i=1,\ldots,n^2\right\}, \mathcal{C}_{\bm w} = \left\{ \bm w \mid w_j \geq 0 \text{ for } j=1,\ldots,p\right\},$ $\mu$ is a positive parameter that is used to balance the two terms in the objective function, ${\rm TV}$ denotes the TV regularization term, $\bm x \in \mathbb{R}^{n^2}$ is a vectorized version of the unknown $n\times n$ image to be recovered, $\bm w \in \mathbb{R}^p$ denotes unknown weights parameterizing the blurring matrix $A\left(\bm w\right) \in \mathbb{R}^{n^2 \times n^2}$, and $\bm d \in \mathbb{R}^{n^2}$ is the observed, blurred noisy image (data). For the imaging model we use, the blurring matrix $A\left(\bm w\right)$ is a weighted sum of $p$ known blurring matrices $A_j$ with the form
\begin{equation*}
A(\bm w) = \displaystyle{\sum_{j=1}^p} w_j A_j  = w_1 A_1 + w_2 A_2 + \cdots + w_p A_p.
\end{equation*}

Note that the fidelity term in the objective function in (\ref{eq:objFctn1}) is nonconvex and the TV regularization term is nondifferentiable and nonlinear. This poses some computational challenges in the optimization problem. In particular, due to the existence of the parameters $\bm w$ and the nonconvexity of the objective function, approaches such as the fast iterative shrinkage-thresholding algorithm (FISTA) \cite{Beck2009, Chen2019} cannot be applied directly to solve the problem. However, alternative approaches have been proposed to solve the optimization problem with TV regularization in the literature. One particularly efficient scheme is the alternating direction method of multipliers (ADMM) \cite{Wang2008, Tao2009, Zhang2017}. ADMM was first developed to solve convex optimization problems by breaking them into subproblems, each of which are then easier to handle.
Recent work has shown that ADMM can also perform well for a variety of applications involving nonconvex objective functions or nonconvex sets \cite{Boyd2011, Wang2019, Lai2014, Xu2012, Yang2017}. Inspired by the success of ADMM on nonconvex problems, we consider using it as the optimization method to tackle problem \eqref{eq:objFctn1}.

The main contribution of this paper can be summarized as follows:
\begin{itemize}
\item We implement an efficient algorithm called ADMM-LAP for myopic deconvolution problems with TV regularization arising from the AO retinal image restoration. Specifically, we first apply ADMM as an outer optimization method to tackle the TV regularizer and then apply the Linearize And Project (LAP) method as an inner optimization method to solve the tightly coupled subproblems arising within each ADMM iteration.
We establish the subsequence convergence of ADMM-LAP to a stationary point and conduct a complexity analysis of ADMM-LAP to demonstrate its computational efficiency for myopic deconvolution problems with TV regularization.
\item
We present extensive numerical experiments to illustrate the effectiveness of the ADMM-LAP method. In addition, we compare the performance of ADMM-LAP with a benchmark method called ADMM-BCD where ADMM is applied to tackle the TV regularization while block coordinate descent (BCD) is applied to solve the coupled subproblems. Compared to ADMM-BCD, ADMM-LAP converges faster and can reach smaller relative errors for both the restored image and the obtained parameters.
\end{itemize}

The paper is organized as follows. In Section 2, we
introduce a general formulation of the AO retinal imaging problem. In Section 3, we first briefly review the iteration format of ADMM  and introduce a variation of LAP for our problem, then propose the ADMM-LAP method and present the convergence results. A benchmark method ADMM-BCD is also discussed. The computational complexity of both methods are analyzed in this section. Numerical results are given in Section 4, and concluding remarks are drawn in Section 5.

\section{Retinal imaging problem}
\label{sec:2}
\subsection{AO retinal imaging model}
\label{ssec:ImageModel}
Adaptive optics (AO) is a well-known optoelectronic technique that compensates for the time-varying aberrations of the eye \cite{Blanco20112}. However, AO flood imaging suffers from an intrinsic limitation that leads to a loss in resolution because the object is three-dimensional, the image contains information from both the in-focus plane and the out-of-focus planes of the object. Hence, interpretation of such images requires an appropriate post-processing, including image deconvolution. In this paper, we focus on the myopic image deconvolution problem that arises from the AO retinal image restoration and propose the efficient ADMM-LAP method to solve the problem.

First, we describe the structure of the adaptive optics (AO) retinal imaging model. In the continuous setting, retinal imaging is typically modeled as a three-dimensional (3D) convolution~\cite{Blanco2011, Blanco2014}:

\begin{equation}
\label{eq:ImageModel}
\bm d_{3D} = \bm h_{3D} \ast_{3D} \bm x_{3D} + \bm e,
\end{equation}

\noindent where $\bm d_{3D} $ is the observed image, $\bm h_{3D}$ is the PSF, $\ast_{3D}$ is the three-dimensional convolution operator, $\bm x_{3D}$ is the true object, and $\bm e$ is the additive noise. If the true object is assumed to be shift invariant along the optical axis, i.e., the $z$-axis, then $\bm x_{3D}$ becomes separable with

\begin{equation*}
\bm x_{3D}\left(x,y,z\right) = \bm x_{2D}\left(x,y\right) \bm w\left(z\right),
\end{equation*}

\noindent where $\bm w\left(z\right)$ is the normalized flux emitted by the plane at depth $z$ such that $\int \bm w(z) dz = 1$. For reasonable optical setups, this shift invariance along the optical axis can be guaranteed to a sufficient degree to make this separability assumption meaningful~\cite{Blanco2011}.

In practice, retinal flood imaging systems typically image along a single plane of interest, a departure from the 3D model in~(\ref{eq:ImageModel}). This results in a 2D data image taken at a single depth. For depth $z = 0$, this gives the observed image

\begin{equation*}
\bm d_{2D}\left(x,y\right) = \bm d_{3D}\left(x,y,0\right).
\end{equation*}

\noindent The shift invariance assumption for the true image $\bm x_{3D}$ then implies that the two-dimensional PSF for the observed image $\bm d_{2D}(x,y)$ can be expressed as

\begin{equation*}
\bm h_{2D}\left(x,y\right) = \int \bm w\left(-z\right) \bm h_{3D}\left(x,y,z\right) dz.
\end{equation*}

Discretizing the above integral using a quadrature rule (a simple rectangle rule is often sufficient) we obtain the PSF for the observed retinal flood images as

\begin{equation*}
\bm h_{2D}\left(x,y\right) \approx \sum_{j=1}^{p} w_j \bm h_j \left(x,y\right),
\end{equation*}

\noindent where $\bm h_j \left(x,y\right) = \bm h_{3D}\left(x,y,z_j\right)$ is the 2D PSF taken at depth $z_j$ and $w_j =\bm w\left(z_j\right) \Delta z_j$ are weights with $\Delta z_j$ the thickness of the $j$th slice of the 3D image in the quadrature sum. Thus, in our myopic deconvolution model (\ref{eq:objFctn1}), the problem becomes determining the unknown weights $w_j$ that parameterize the PSF with the constraints that $\sum_{j=1}^{p} w_j = 1$ and $w_j \geq 0$ for all $j$. We point out that it can be seen from numerical experiments that the global PSF in (\ref{eq:objFctn1}) is normalized (i.e., all entries sum to approximately 1). Therefore, the myopic deconvolution model satisfies the fact that the PSF is energy preserving.

\subsection{Myopic deconvolution model with TV regularization}
For TV regularization in (\ref{eq:objFctn1}), we follow the notations used in \cite{Wang2008, Tao2009}: the discrete form of TV for a grayscale image $\bm x\in \mathbb{R}^{n^2}$ is defined as
\begin{equation}
\label{TVdef}
{\rm TV}\left(\bm x\right)=\sum_{i=1}^{n^2}\|D_{i}\bm x\|_{2},
\end{equation}
where for each $i$, $D_{i}\bm x\in\mathbb{R}^{2}$ represents the first-order finite difference of $\bm x$ at pixel $i$ in both horizontal and vertical directions. We note that the 2-norm in (\ref{TVdef}) can be replaced by the 1-norm. If the 2-norm is used, then we obtain the isotropic version of  TV, and if the 1-norm is used, we obtain the anisotropic version. In this paper, we will only treat the isotropic case for simplicity, but the treatment for the anisotropic case is completely analogous.

\textbf{Notations}: The two first-order global finite difference operators in horizontal and vertical directions are, respectively, denoted by $D^{(1)}, D^{(2)}\in \mathbb R^{n^2\times n^2}$.
$D_{i}\in\mathbb R^{2\times n^2}$ is a two-row matrix formed by stacking the $i$th row of $D^{(1)}$ on top of the $i$th row  of $D^{(2)}$ and
$D:=\left(D^{(1)};D^{(2)}\right)\in \mathbb R^{2n^2\times n^2}$ is the global first-order finite difference operator.

In order to deal with the restriction on the sum of the weights, we introduce an appropriate regularizer on $\bm w$:
\begin{equation}
\label{eq:objFctn1*}
\min_{\bm x \in \mathcal{C}_{\bm x}, \bm w \in {\mathcal{C}}_{\bm w}}  \frac{\mu}{2} \|{ A}\left(\bm w\right) \bm x - \bm d \|_2^2+\sum_{i=1}^{n^2}\|D_{i}\bm x\|_{2}+S(\bm w)
\end{equation}
where $\mathcal{C}_{\bm x} = \left\{\bm  x \mid x_i \geq 0 \text{ for } i=1,\ldots,n^2\right\},
{\mathcal{C}}_{\bm w} = \left\{ \bm w \mid w_j \geq 0 \text{ for } j=1,\ldots,p\right\}$ and $A(\bm w) = \displaystyle{\sum_{j=1}^p} w_j A_j $. The regularizer $S(\bm w)$ is defined as
\begin{equation*}
S(\bm w) = \frac{\xi}{2}(\bm {e}^{\top} \bm w - 1)^2,
\end{equation*}
\noindent where $\bm {e} \in \mathbb{R}^p$ is the vector all ones and $\xi > 0$ is a weighting parameter.
This regularizer penalizes solutions where the weights $w_k$ do not sum to $1$, so for appropriately large $\xi$, it can effectively enforce the summation constraint. We choose this option because it fits conveniently within the LAP framework in Section \ref{sec3.2}. Another possible alternative is a Lagrangian multiplier approach, which would strictly enforce the summation \cite{Nocedal1999}.

\section{Optimization schemes}
In this section, we first briefly review the alternating direction method of multipliers (ADMM), discuss how to adapt ADMM to solve \eqref{eq:objFctn1*} and discuss how to adapt a variation of the Linearize And Project (LAP) method to solve the related coupled subproblems. We then propose the ADMM-LAP method for solving \eqref{eq:objFctn1*} and give the convergence analysis. A benchmark method called  ADMM-BCD is also introduced in this section. We compare their computational complexity and show that ADMM-LAP is more efficient when solving \eqref{eq:objFctn1*}.

\subsection{ADMM splitting}
The classical ADMM is designed to solve the following 2-block optimization problem with linear constraints
\begin{equation}
\begin{aligned}
\min \limits_{\bm u,\bm v\in \mathbb{R}^{n^2}}^{}&f\left(\bm u\right)+g\left(\bm v\right) \\
{\rm s.t.}\ \ &B\bm u+C\bm v=\bm z,
\end{aligned}
\label{eq:admm1}
\end{equation}
where $f\left(\bm u\right):\mathbb{R}^{n^2}\rightarrow (-\infty, +\infty]$, $g\left(\bm v\right):\mathbb{R}^{n^2}\rightarrow (-\infty, +\infty]$ are convex functions, and $B,C\in \mathbb{R}^{n^{2}\times n^{2}}$ and $\bm z \in \mathbb{R}^{n^2}$ are given.

The augmented Lagrangian function for \eqref{eq:admm1} is given by
\begin{equation*}
 \mathcal{L}_{\beta}\left(\bm u,\bm v,\bm \lambda\right)=f\left(\bm u\right)+g\left(\bm v\right)+\left(\bm \lambda,B\bm u+C\bm v-\bm z\right)+\frac{\beta}{2}\|B\bm u+C\bm v-\bm z\|^2_{2},
\end{equation*}
where $\bm \lambda \in \mathbb{R}^{n^2}$ denotes the Lagrange multiplier, $\beta>0$ is the penalty parameter.

Given $\left(\bm u^0,\bm v^0,\bm \lambda^0\right)\in \mathbb{R}^{n^2}\times\mathbb{R}^{n^2}\times \mathbb{R}^{n^2}$, the penalty parameter $\beta>0$,  ADMM iterates as follows:
\begin{equation*}
\label{ADMM}
        \left\{\begin{aligned}
\bm u^{k+1}&=\underset{\bm u}{\text{argmin}}\mathcal{L}_{\beta}\left(\bm u,\bm v^k,\bm \lambda^k\right),\\
\bm v^{k+1}&=\underset{\bm v}{\text{argmin}}\mathcal{L}_{\beta}\left(\bm u^{k+1},\bm v,\bm \lambda^k\right),\\
\bm \lambda^{k+1}&=\bm \lambda^k-\beta\left(B\bm u^{k+1}+C\bm v^{k+1}-\bm z\right).
                          \end{aligned} \right.
\end{equation*}

In order to apply ADMM to solve (\ref{eq:objFctn1*}), we introduce artificial vectors
$\mathbf y_{i}\in\mathbb{R}^{2}, \ i=1,\dots,n^2$,  then we can rewrite (\ref{eq:objFctn1*}) in an equivalent form:

\begin{equation}
\label{eq:objFctn*}
\begin{aligned}
\min_{\bm x , \bm w, \bm y}
\frac{\mu}{2} \| &A(\bm w) \bm x - \bm d \|_2^2 + S(\bm w) +\sum_{i=1}^{n^2}\|\mathbf y_{i}\|_{2}+\delta_{\mathcal{C}_{\bm x}}(\bm x)+\delta_{{\mathcal{C}}_{\bm  w}}(\bm w),\\
\text{s.t.} \ \ \quad  &\mathbf y_{i}=D_{i}\bm x,\qquad i=1,...,n^2,
\end{aligned}
\end{equation}
where $\mathcal{C}_{\bm x} = \left\{\bm  x \mid x_i \geq 0 \text{ for } i=1,\ldots,n^2\right\},
{\mathcal{C}}_{\bm w} = \left\{ \bm w \mid w_j \geq 0 \text{ for } j=1,\ldots,p\right\}$, $A(\bm w) = \displaystyle{\sum_{j=1}^p} w_j A_j $ and $\delta_{\mathcal{C}}(\cdot)$ denotes the indicator function of ${\mathcal{C}}$, i.e.,
\begin{equation*}
\delta_{\mathcal{C}}(s)=\left\{\begin{array}{ll}{0,} & {s \in {\mathcal{C}}}, \\ {\infty,} & {s \notin {\mathcal{C}}}.\end{array}\right.
\end{equation*}

For convenience, let
\footnote[1]{Borrowing MATLAB notation, we use the semicolon in a vector to denote concatenation of terms.} $\bm y=[\bm y_{1};\bm y_{2}]\in\mathbb{R}^{2n^2}$, where $\bm y_{1}, \bm y_{2}$ are vectors of length $n^2$ and $[(\bm y_{1})_{i};(\bm y_{2})_{i}]=\mathbf y_{i}\in\mathbb{R}^{2}$ for $i=1,...,n^2$.
The augmented Lagrangian function for \eqref{eq:objFctn*} is then given by
\begin{eqnarray*}
\label{eq:lag1}
\mathcal{L}_{\beta}\left(\bm x, \bm w, \bm y,\bm \lambda\right) & :=& \frac{\mu}{2} \| A(\bm w) \bm x - \bm d \|_2^2+S(\bm w)+\delta_{\mathcal{C}_{\bm x}}(\bm x)+\delta_{{\mathcal{C}}_{\bm  w}}(\bm w) \\
& &+ \sum_{i=1}^{n^2}\left(\|\mathbf y_{i}\|_{2}-\bm \lambda_{i}^{T}\left(\mathbf y_{i}-D_{i}\bm x\right)+\frac{\beta}{2}\|\mathbf y_{i}-D_{i}\bm x\|^{2}_{2}\right), \nonumber
\end{eqnarray*}
where each $\bm \lambda_{i}\in \mathbb{R}^{2}$  and $\bm \lambda \in \mathbb{R}^{2n^2}$ is a reordering of $\bm \lambda_{i}$ similar to $\bm y$.

Consider $\left(\bm x,\bm w\right)$ as one block of variables and $\bm y$ as the other. We can now apply ADMM as the outer optimization method to tackle the TV regularization term in \eqref{eq:objFctn*}. Given the initial point $\left(\bm y^{0}, \bm x^{0}, \bm w^{0}, \bm \lambda^{0}\right)$,  the ADMM algorithm iteratively solves the following three subproblems:
\begin{equation*}
\begin{cases}
\begin{aligned}
\bm y^{k+1}&=\underset{\bm y}{\rm argmin}\mathcal{L}_{\beta}\left(\bm x^{k},\bm w^{k},\bm y,\bm \lambda^{k}\right),\\
(\bm x^{k+1},\bm w^{k+1})&=\underset{\bm x , \bm w}{\rm argmin}\mathcal{L}_{\beta}\left(\bm x,\bm w,\bm y^{k+1},\bm \lambda^{k}\right),\\
\bm \lambda^{k+1}&=\bm \lambda^{k}-\beta \left(\bm y^{k+1}-D\bm x^{k+1}\right).
\end{aligned}
\end{cases}
\end{equation*}

First, for the $\bm y$-subproblem, notice that minimizing $\mathcal{L}_{\beta}\left(\bm x^{k},\bm w^{k},\bm y,\bm \lambda^{k}\right)$ with respect to $\bm y$
is equivalent to minimizing $n^2$ two-dimensional problems of the form
\begin{equation}
\min_{\mathbf y_{i}\in\mathbb{R}^{2}}\|\mathbf y_{i}\|_{2}+\frac{\beta}{2}\left\|\mathbf y_{i}-\left(D_{i}\bm x^{k}+\frac{1}{\beta}\left(\bm \lambda^{k}\right)_{i}\right)\right\|^{2}_{2}, \quad i=1,2,...,n^2.
\label{eq:suby}
\end{equation}
The solution to \eqref{eq:suby} is given explicitly by the two-dimensional shrinkage \cite{Tao2009}
\begin{equation}
\label{ysubproblem}
\mathbf y_{i}^{k+1}={\rm max}\left\{\left\|D_{i}\bm x^{k}+\frac{1}{\beta}\left(\bm \lambda^{k}\right)_{i}\right\|_{2}-\frac{1}{\beta}, 0\right\}\frac{D_{i}\bm x^{k}+\frac{1}{\beta}\left(\bm \lambda^{k}\right)_{i}}{\|D_{i}\bm x^{k}+\frac{1}{\beta}\left(\bm \lambda^{k}\right)_{i}\|_{2}}, \ i=1,2,...,n^2.
\end{equation}

Second, denote
\begin{equation*}
R(\bm x, \bm y,\bm \lambda):=\sum_{i=1}^{n^2} \left(\|\mathbf y_{i}\|_{2} -\bm \lambda_i^{\top} \left(\mathbf y_{i} - D_i \bm x\right) + \frac{\beta}{2} \|\mathbf y_{i} - D_i \bm x \|^2_2 \right),
\end{equation*}
and define
\begin{equation*}
\hat{\Phi}(\bm x, \bm w, \bm y,\bm \lambda) := \frac{\mu}{2} \| A\left(\bm w\right) \bm x - \bm d \|_2^2+S(\bm w)+ R\left(\bm x, \bm y,\bm \lambda\right).
\end{equation*}
Then it can be shown that the $(\bm x,\bm w)$-subproblem is equivalent to the following problem:
\begin{equation}
\left(\bm x^{k+1},\bm w^{k+1}\right)
=\underset{\bm x \in \mathcal{C}_{\bm x}, \bm w \in {{\mathcal{C}}}_{\bm w}}{\rm argmin}\hat{\Phi}(\bm x, \bm w, \bm y ^{k+1},\bm \lambda ^{k}).
\label{xw:subproblem}
\end{equation}
Note that this is a tightly coupled optimization problem with element-wise bound constraints.

Finally, update the multiplier,
\begin{equation}
\label{eq:lambda}
\bm \lambda^{k+1}=\bm \lambda^{k}-\beta \left(\bm y^{k+1}-D\bm x^{k+1}\right).
\end{equation}
If the termination criterion is met, stop; else, set $k:=k+1$ and go to the $\bm y$-subproblem.

\subsection{LAP method}
\label{sec3.2}
To solve the tightly coupled $(\bm x,\bm w)$-subproblem (\ref{xw:subproblem}), we develop a variation of the Linearize And Project (LAP) method proposed by Herring \textit{et al}~\cite{Herring2017}. The LAP method is efficient for inverse problems with multiple, tightly coupled blocks of variables such as the problem under consideration. Its strengths include the option to impose element-wise bound constraints on all blocks of variables.

In this paper, we present the LAP method based on the normal equation approach instead of the least squares approach presented in the original paper \cite{Herring2017}. First, we consider the unconstrained problem where $\mathcal{C}_{\bm x}=\mathbb{R}^{n^2}$, ${\mathcal{C}}_{\bm w}=\mathbb{R}^{p}$. Denote the residual as $\bm r\left(\bm x,\bm w\right) := A\left(\bm w\right) \bm x - \bm d$. Then at the iterate $(\bm x,\bm w)$, the Jacobian with respect to the $\bm x$ block of variables is
\begin{equation}
\label{jx}
{\mathbf{J}}_{\bm x} = \nabla_{\bm x} \bm r(\bm x,\bm w)^\top  = \bm A(\bm w)^\top,
\end{equation}
and the Jacobian with respect to the $\bm w$ block of variables is
\begin{equation}
\label{jw}
\bm J_{\bm w} = \begin{bmatrix} (A_1 \bm x)^{\top} \\ (A_2 \bm x)^{\top} \\ \vdots \\ (A_p \bm x)^{\top} \end{bmatrix}.
\end{equation}

Computing the update step around the current iterate $\left(\bm x, \bm w\right)$ requires the gradient and Hessian of the objective function $\hat{\Phi}(\bm x, \bm w, \bm y^{k+1},\bm \lambda^{k})$. These are given by

\begin{align}
\nabla_{\bm x,\bm w} \hat{\Phi}(\bm x, \bm w, \bm y ^{k+1},\bm \lambda ^{k})  &= \mu\begin{bmatrix} {\mathbf{J}}_{\bm x}^{\top} \bm r \\ {\mathbf{J}}_{\bm w}^{\top} \bm r \end{bmatrix} + \begin{bmatrix} \nabla_{\bm x} R\left(\bm x, \bm y ^{k+1},\bm \lambda ^{k}\right) \\  \nabla_{\bm w} S(\bm w)  \end{bmatrix}, \\
\nabla^2_{\bm x,\bm w} \hat{\Phi}(\bm x, \bm w, \bm y ^{k+1},\bm \lambda ^{k})  &\approx  \mu\begin{bmatrix} {\mathbf{J}}_{\bm x}^{\top} {\mathbf{J}}_{\bm x} & {\mathbf{J}}_{\bm x}^{\top} {\mathbf{J}}_{\bm w} \\ {\mathbf{J}}_{\bm w}^{\top} {\mathbf{J}}_{\bm x}  & {\mathbf{J}}_{\bm w}^{\top} {\mathbf{J}}_{\bm w} \end{bmatrix} +
\begin{bmatrix} \nabla^2_{\bm x} R\left(\bm x, \bm y ^{k+1},\bm \lambda ^{k}\right) &\bm 0 \\\bm 0 &\nabla_{\bm w}^{2} S(\bm w) \end{bmatrix},
\label{eq:linear}
\end{align}
where $\bm r:=\bm r(\bm x, \bm w)$. Here, the Hessian of the regularizer $R(\bm x)$ can be exact or a linearized approximation.
Let $\delta {\bm x}$ and $\delta \bm w$ denote the update step for the image and the parameters, respectively.
Then the update steps $\delta {\bm x}$ and $\delta \bm w$ are given by the solution of the following block linear system
\begin{equation}
\label{eq:blockLAP}
\begin{bmatrix} {\mathbf{J}}_{\bm x}^{\top} {\mathbf{J}}_{\bm x} +\frac{1}{\mu} \nabla^2_{\bm x} R\left(\bm x, \bm y ^{k+1},\bm \lambda ^{k}\right) & {\mathbf{J}}_{\bm x}^{\top} {\mathbf{J}}_{\bm w} \\ {\mathbf{J}}_{\bm w}^{\top} {\mathbf{J}}_{\bm x}  & {\mathbf{J}}_{\bm w}^{\top} {\mathbf{J}}_{\bm w}+\nabla_{\bm w}^2 S(\bm w) \end{bmatrix}
\begin{bmatrix} \delta {\bm x} \\ \delta \bm w \end{bmatrix} =
-\begin{bmatrix} {\mathbf{J}}_{\bm x}^{\top} \bm r +\frac{1}{\mu} \nabla_{\bm x} R\left(\bm x, \bm y ^{k+1},\bm \lambda ^{k}\right) \\ {\mathbf{J}}_{\bm w}^{\top} \bm r +\nabla_{\bm w} S(\bm w)\end{bmatrix}.
\end{equation}

Note that omitting the regularizer terms, (\ref{eq:blockLAP}) is the normal equation corresponding to the least-squares problem
\begin{equation*}
\frac{\mu}{2} \left\| \begin{bmatrix} \mathbf{J}_{\bm x} & \mathbf{J}_{\bm w} \end{bmatrix} \begin{bmatrix} \delta \bm x \\ \delta \bm w \end{bmatrix} + \bm r\right\|^2_2
\end{equation*}
that can be obtained by the \emph{Linearize} step of LAP in \cite{Herring2017}.

LAP solves (\ref{eq:blockLAP}) by projecting the original problem onto a reduced space. In this paper, we choose to project the problem onto the image space, i.e., we eliminate the block of variables corresponding to $\bm w$. When projecting the problem onto the image space, $\delta \bm w$ can be computed by
\begin{equation}
\label{delta w1}
\delta \bm w = - \left({\mathbf{J}}_{\bm w}^{\top} {\mathbf{J}}_{\bm w}+\nabla_{\bm w}^2 S(\bm w)\right)^{-1} \left( {\mathbf{J}}_{\bm w}^{\top} {\mathbf{J}}_{\bm x} \delta \bm x + {\mathbf{J}}_{\bm w}^{\top} \bm r+\nabla_{\bm w} S(\bm w)\right).
\end{equation}
Plug \eqref{delta w1} into \eqref{eq:blockLAP} and get
\begin{equation*}
\begin{aligned}
&\left({\mathbf{J}}_{\bm x}^{\top} {\mathbf{J}}_{\bm x} + \frac{1}{\mu} \nabla^2_{\bm x} R\left(\bm x, \bm y ^{k+1},\bm \lambda ^{k}\right)\right) \delta \bm x
- {\mathbf{J}}_{\bm x}^{\top} {\mathbf{J}}_{\bm w} \left({\mathbf{J}}_{\bm w}^{\top} {\mathbf{J}}_{\bm w}+\nabla_{\bm w}^2 S(\bm w)\right)^{-1} \left( {\mathbf{J}}_{\bm w}^{\top} {\mathbf{J}}_{\bm x} \delta \bm x + {\mathbf{J}}_{\bm w}^{\top} \bm r+\nabla_{\bm w} S(\bm w)\right)\\
= &-\left({\mathbf{J}}_{\bm x}^{\top} \bm r +\frac{1}{\mu} \nabla_{\bm x} R\left(\bm x, \bm y ^{k+1},\bm \lambda ^{k}\right)\right),
\end{aligned}
\end{equation*}
which can be simplified as
\begin{equation}
\label{delta x}
 \bm { M} \delta \bm x
= \bm {b},
\end{equation}
the operator and the right hand side are given by
\begin{equation*}
\begin{aligned}
\bm { M}  :=&
{\bm J}_{\bm x}^{\top}(\bm I - {\mathbf{J}}_{\bm w} ({\mathbf{J}}_{\bm w}^{\top} {\mathbf{J}}_{\bm w} + \nabla^2_{\bm w} S(\bm w))^{-1} {\mathbf{J}}_{\bm w}^{\top}) {\mathbf{J}}_{\bm x} +\frac{1}{\mu}\nabla_{\bm x}^{2} R\left(\bm x, \bm y ^{k+1},\bm \lambda ^{k}\right),\\
\bm {b}  :=&   -{\mathbf{J}}_{\bm x}^{\top}(\bm I - {\mathbf{J}}_{\bm w} ({\mathbf{J}}_{\bm w}^{\top} {\mathbf{J}}_{\bm w} + \nabla^2_{\bm w} S(\bm w))^{-1} {\mathbf{J}}_{\bm w}^{\top}) \bm r - \frac{1}{\mu}\nabla_{\bm x} R\left(\bm x, \bm y ^{k+1},\bm \lambda ^{k}\right) \\
&+
{\mathbf{J}}_{\bm x}^{\top} {\mathbf{J}}_{\bm w} ({\mathbf{J}}_{\bm w}^{\top} {\mathbf{J}}_{\bm w} + \nabla^2_{\bm w} S(\bm w))^{-1} \nabla_{\bm w} S(\bm w).
\end{aligned}
\end{equation*}

Moreover, it is easy to see that the gradient and Hessian of  $R(\bm x, \bm y ^{k+1},\bm \lambda ^{k})$ satisfy
\begin{equation}
\begin{aligned}
\label{Rgradient}
\nabla_{\bm x} R\left(\bm x, \bm y ^{k+1},\bm \lambda ^{k}\right) &= \sum_{i=1}^{n^2}\left(D_{i}^{T} (\bm \lambda^{k})_{i}-\beta D_{i}^{T}\left(\mathbf y_{i}^{k+1}-D_{i}\bm x\right)\right)\\
&=D^{T}\bm \lambda^{k}-\beta D^{T}\bm y^{k+1}+\beta D^{T}D\bm x,
\end{aligned}
\end{equation}
and
\begin{equation}
\begin{aligned}
\label{RHessian}
\nabla^2_{\bm x} R\left(\bm x, \bm y ^{k+1},\bm \lambda ^{k}\right) &= \sum_{i=1}^{n^2}\beta D_{i}^{T}D_{i}=\beta D^{T}D.
\end{aligned}
\end{equation}
Following the above procedures, we are able to compute the unconstrained update steps $\delta \bm x$ and  $\delta \bm w$.

Next, we consider modifying the above procedures in order to handle the element-wise bound constraints on $\bm x$ and $\bm w$.
We know that a simple extension to the Gauss-Newton method does not work \cite{Kelley1999}, however, a simple correction, i.e. the projected Gauss-Newton method \cite{Haber2014} can be made convergent and it works well for many inverse problems. To this end, the variables are divided into active set variables and inactive set variables and the step $\delta \bm x$ and $\delta \bm w$ are computed through these two separate sets.
Let the feasible index set be defined as
\begin{equation*}
\mathcal{N}:= \left\{ q\in \mathbb{N} \ \bigg| \left[\begin{array}{c}{\bm x_{q}} \\ {\bm w_{q}}\end{array}\right] \geq \bm 0\right\}.
\end{equation*}
Define the active and inactive sets as
\begin{equation*}
\begin{aligned}
&\mathcal{A}:= \left\{ q\in \mathbb{N} \ \bigg| \left[\begin{array}{c}{\bm x_{q}} \\ {\bm w_{q}}\end{array}\right] = \bm 0 \right\},\\
&\mathcal{I}:= \mathcal{N}\backslash \mathcal{A}.
\end{aligned}
\end{equation*}
Then, we can divide the variables into the active set variables $\left[\begin{array}{c}{\bm x_{\mathcal{A}}} \\ {\bm w_{\mathcal{A}}}\end{array}\right]$
and the inactive set variables $\left[\begin{array}{c}{\bm x_{\mathcal{I}}} \\ {\bm w_{\mathcal{I}}}\end{array}\right]$. Denote the steps taken on $\left[\begin{array}{c}{\bm x_{\mathcal{A}}} \\ {\bm w_{\mathcal{A}}}\end{array}\right]$
by $\left[\begin{array}{c}{\delta \bm x_{\mathcal{A}}} \\ {\delta \bm w_{\mathcal{A}}}\end{array}\right]$
and the steps taken on $\left[\begin{array}{c}{\bm x_{\mathcal{I}}} \\ {\bm w_{\mathcal{I}}}\end{array}\right]$ by $\left[\begin{array}{c}{\delta \bm x_{\mathcal{I}}} \\ {\delta \bm w_{\mathcal{I}}}\end{array}\right]$, respectively.

$\delta \bm x_{\mathcal{I}}$ and $\delta \bm w_{\mathcal{I}}$ can be computed in a similar way as the unconstrained case except that the variables need to be projected onto the inactive set. That is, $\delta {\bm x}_{\mathcal{I}}$ at the current iterate $\left(\bm x,\bm w\right)$ is computed as:
\begin{equation}
\label{eq:rdp1*}
\bm {\hat M} \delta {\bm x}_{\mathcal{I}}
= \bm {\hat b},
\end{equation}
the operator and right hand side are given by
\begin{equation*}
\begin{aligned}
\bm {\hat M}  :=&
\mathbf{\hat J}_{\bm x}^{\top}(\bm I - \mathbf {\hat J}_{\bm w} (\mathbf {\hat J}_{\bm w}^{\top} \mathbf {\hat J}_{\bm w} + \nabla^2_{\bm w} {\hat S}(\bm w))^{-1} \mathbf {\hat J}_{\bm w}^{\top}) \mathbf {\hat J}_{\bm x} +\frac{1}{\mu}\nabla_{\bm x}^{2} {\hat R}\left(\bm x, \bm y ^{k+1},\bm \lambda ^{k}\right),\\
\bm {\hat b}  :=&   -\mathbf {\hat J}_{\bm x}^{\top}(\bm I - \mathbf {\hat J}_{\bm w} (\mathbf {\hat J}_{\bm w}^{\top} \mathbf {\hat J}_{\bm w} + \nabla^2_{\bm w} {\hat S}(\bm w))^{-1} \mathbf {\hat J}_{\bm w}^{\top}) \bm r - \frac{1}{\mu}\nabla_{\bm x} {\hat R}\left(\bm x, \bm y ^{k+1},\bm \lambda ^{k}\right) \\
&+
\mathbf {\hat J}_{\bm x}^{\top} \mathbf {\hat J}_{\bm w} (\mathbf {\hat J}_{\bm w}^{\top} \mathbf {\hat J}_{\bm w} + \nabla^2_{\bm w} {\hat S}(\bm w))^{-1} \nabla_{\bm w} {\hat S}(\bm w),
\end{aligned}
\end{equation*}
where $\hat{\mathbf{J}}_{\bm x}$, $\hat{\mathbf{J}}_{\bm w}$, $\nabla_{\bm x} \hat{R}$, $\nabla^2_{\bm x} {\hat R}$, $\nabla_{\bm w} {\hat S}$ and $\nabla_{\bm w}^{2} {\hat S}$ represent ${\mathbf{J}}_{\bm x}$, ${\mathbf{J}}_{\bm w}$, $\nabla_{\bm x} R$, $\nabla^2_{\bm x} {R}$, $\nabla_{\bm w} {S}$ and $\nabla_{\bm w}^{2} {S}$  restricted to the inactive set via projection, respectively.
The reduced problem \eqref{eq:rdp1*} does not need to be solved to a high accuracy. For example in~\cite{Herring2017},  a stopping tolerance of $10^{-1}$ is used to solve the reduced problem iteratively. After solving for $\delta \bm x_{\mathcal{I}}$, $\delta \bm w_{\mathcal{I}}$ can be computed by
\begin{equation}
\label{delta w}
\delta \bm w_{\mathcal{I}} = - \left({\mathbf{\hat J}}_{\bm w}^{\top} {\mathbf{\hat J}}_{\bm w}+\nabla_{\bm w}^2 {\hat S}(\bm w)\right)^{-1} \left( {\mathbf{\hat J}}_{\bm w}^{\top} {\mathbf{\hat J}}_{\bm x} \delta \bm x_{\mathcal{I}} + {\mathbf{\hat J}}_{\bm w}^{\top} \bm r+\nabla_{\bm w} {\hat S}(\bm w)\right).
\end{equation}

For the active set, $\delta \bm x_{\mathcal{A}}$ and ${\delta \bm w_{\mathcal{A}}}$ are given by a scaled projected gradient descent step
\begin{equation}
\label{equ:activeset}
\left[\begin{array}{c}{\delta \bm x_{\mathcal{A}}} \\ {\delta \bm w_{\mathcal{A}}}\end{array}\right]=-\mu\left[\begin{array}{c}{\tilde{\mathbf{J}}_{\bm x}^{\top} \bm r} \\ {\tilde{\mathbf{J}}_{\bm w}^{\top} \bm r}\end{array}\right]-\left[\begin{array}{c}{\nabla_{\bm x} \tilde{R}\left(\bm x+\delta {\bm x}_{\mathcal{A}}, \bm y ^{k+1},\bm \lambda ^{k}\right)} \\ \nabla_{\bm w} {\tilde S}(\bm w)\end{array}\right],
\end{equation}
where $\tilde{\mathbf{J}}_{\bm x}$, $\tilde{\mathbf{J}}_{\bm w}$, $\nabla_{\bm x}\tilde{R}$ and $\nabla_{\bm w} {\tilde S}$ represent the projection of ${\mathbf{J}}_{\bm x}$, ${\mathbf{J}}_{\bm w}$, $\nabla_{\bm x} {R}$ and $\nabla_{\bm w} {S}$ onto the active set, respectively.

Then $\delta \bm x$ and $\delta \bm w$ can be calculated as a scaled combination of $\delta \bm x_{\mathcal{A}}$, ${\delta \bm w_{\mathcal{A}}}$, $\delta \bm x_{\mathcal{I}}$ and ${\delta \bm w_{\mathcal{I}}}$ by
\begin{equation}
\label{equ:step}
\left[\begin{array}{c}{\delta \bm x} \\ {\delta \bm w}\end{array}\right]=\left[\begin{array}{c}{\delta \bm x_{\mathcal{I}}} \\ {\delta \bm w_{\mathcal{I}}}\end{array}\right]+\gamma\left[\begin{array}{c}{\delta \bm x_{\mathcal{A}}} \\ {\delta \bm w_{\mathcal{A}}}\end{array}\right],
\end{equation}
and the parameter $\gamma$ is selected based on the recommendation in \cite{Haber2014},
\begin{equation*}
\label{equ:parameter}
\gamma=\frac{\max \left(\left\|\delta \bm x_{\mathcal{I}}\right\|_{\infty},\left\|\delta \bm w_{\mathcal{I}}\right\|_{\infty}\right)}{\max \left(\left\|\delta \bm x_{\mathcal{A}}\right\|_{\infty},\left\|\delta \bm w_{\mathcal{A}}\right\|_{\infty}\right)}.
\end{equation*}
Finally, the projected Armijo line search is applied to find the solution. Here the modified Armijo condition is given by
\begin{equation}
\begin{aligned}
\label{equ:armijo}
&\hat{\Phi}\left(\mathbf{P}_{\mathcal{C}_{\bm x}}\left(\bm x+\eta \delta \bm x\right), \mathbf{P}_{{\mathcal{C}}_{\bm w}}\left(\bm w+\eta \delta \bm w\right), \bm y ^{k+1},\bm \lambda ^{k}\right)\\
 \leq &\hat{\Phi}\left(\bm x, \bm w, \bm y ^{k+1},\bm \lambda ^{k}\right)+c \eta \mathbf{Q}\left(\nabla_{\bm x, \bm w} \hat{\Phi}\left(\bm x, \bm w, \bm y ^{k+1},\bm \lambda ^{k}\right)\right)^{\top}\left[\begin{array}{c}{\delta \bm x} \\ {\delta \bm w}\end{array}\right],
\end{aligned}
\end{equation}
where $\mathbf{P}_{\mathcal{C}_{\bm x}}$ and $\mathbf{P}_{{\mathcal{C}}_{\bm w}}$ denotes the projections onto the feasible set for the image and parameters, respectively, $\mathbf{Q}\left(\nabla_{\bm x, \bm w} \hat{\Phi}\left(\bm x, \bm w, \bm y ^{k+1},\bm \lambda ^{k}\right)\right)$ denotes the projected gradient, $0<\eta\leq 1$ denotes the step size by backtracking, and we set $c=10^{-4}$ as suggested in \cite{Nocedal1999}. We point out that it is necessary to update the inactive set and the active set in each iteration because the projection does not prevent variables from leaving the active set and joining the inactive set. Hence, it makes sense that in Algorithm \ref{alg:ADMM-LAP}, if the termination criterion is not met in Step 3, we go back to Step 2 and get a new solution.

\subsection{ADMM-LAP method}
The proposed ADMM-LAP method for solving myopic deconvolution problems with TV regularization is summarized in Algorithm \ref{alg:ADMM-LAP}.

\begin{algorithm}[H]
\caption{ADMM-LAP method for (\ref{eq:objFctn*})}
\label{alg:ADMM-LAP}
\normalsize
\begin{algorithmic}
\STATE{\textbf{Input:}} $\left(\bm y^{0},\bm x^{0},\bm w^{0},{\bm \lambda}^{0}\right)\in \mathbb{R}^{2n^2}\times \mathbb{R}^{n^2}\times \mathbb{R}^{p}\times\mathbb{R}^{2n^2}$, the penalty parameter $\beta>0$. Let$\lbrace \epsilon_{k+1}\rbrace_{k=0}^{\infty}$ be a sequence satisfying $\lbrace\epsilon_{k+1}\rbrace_{k=0}^{\infty}\subseteq[0,+\infty)$ and $\sum_{k=0}^{\infty}\epsilon_{k+1}<\infty$. Set $k=0$.
\STATE{\textbf{Output:}} $\bm y^{k},\bm x^{k},\bm w^{k},{\bm \lambda}^{k}$.
\STATE{\textbf{Step 1}} Compute $\bm y^{k+1}$ using (\ref{ysubproblem}).
\STATE{\textbf{Step 2}} Find a minimizer of
$$\underset{\bm x \in \mathcal{C}_{\bm x}, \bm w \in {\mathcal{C}}_{\bm w}}{\min}\hat{\Phi}(\bm x, \bm w, \bm y ^{k+1},\bm \lambda ^{k}).
$$
\\
\ \ \ \ \ \ \ \ \ \ Specifically,\\
\ \ \ \ \ \ \ \ \ \ 1.Compute the step on the inactive set by the LAP method using  (\ref{eq:rdp1*}) and \\
\ \ \ \ \ \ \ \ \ \ (\ref{delta w}).
\\
\ \ \ \ \ \ \ \ \ \ 2.Compute the step on the active set by (\ref{equ:activeset}) using the projected gradient \\
\ \ \ \ \ \ \ \ \ \ descent method.\\
\ \ \ \ \ \ \ \ \ \ 3.Combine the steps using (\ref{equ:step}).\\
\ \ \ \ \ \ \ \ \ \ 4.Perform the projected Armijo line search satisfying
(\ref{equ:armijo}) to update $\bm x^{k+1}$, \\
\ \ \ \ \ \ \ \ \ \ $\bm w^{k+1}$.\\
\ \ \ \ \ \ \ \ \ \ 5.Update active and inactive sets.
\STATE{\textbf{Step 3}} If the residual $\left[\begin{array}{c}{\eta_{\bm x}^{k+1}} \\ {\eta_{\bm w}^{k+1}}\end{array}\right]
:= \nabla_{(\bm x, \bm w)} \hat{\Phi}(\bm x, \bm w, \bm y ^{k+1},\bm \lambda ^{k})$ satisfies
$\left\|\left[\begin{array}{c}{\eta_{\bm x}^{k+1}} \\ {\eta_{\bm w}^{k+1}}\end{array}\right]\right\|_{2}\leq \epsilon_{k+1},$
 stop and go to Step 4; else, repeat Step 2.
\STATE{\textbf{Step 4}} Compute $\bm \lambda^{k+1}$ using (\ref{eq:lambda}).
\STATE{\textbf{Step 5}} If a termination criteria is met, stop; else, set $k:=k+1$ and go to Step 1.
\end{algorithmic}
\end{algorithm}

Notice that ADMM-LAP is an inexact ADMM, where the $\bm y$-subproblems are exactly solved while the $(\bm x, \bm w)$-subproblems are inexactly solved. Inspired by the recent results in \cite{Mei2018, Wang2019}, we prove that the ADMM-LAP algorithm converges subsequently to a stationary point.

\begin{theorem}
\label{convergence theorem}
Let $(\bm y^0, \bm x^0, \bm w^0, \bm \lambda^0)$ be any initial point and $\{(\bm y^k, \bm x^k, \bm w^k, \bm \lambda^k)\}$ be the sequence of iterates generated by Algorithm \ref{alg:ADMM-LAP}.
Then if $\beta>\frac{1}{a_{1}}\lambda_{\rm max}(A(\bm w^{k})^{T}A(\bm w^{k}))$ and $\beta$ satisfies $\beta^{2}a_{1}-L\beta-4a_{2}C_{0}>0$, where $C_{0}$, $a_{1}$, $a_{2}$, $L$ are constants specified in Lemma \ref{lemma1} and Lemma \ref{lemma2}, Algorithm \ref{alg:ADMM-LAP} converges subsequently, i.e., it generates a sequence that has a convergent subsequence, whose limit $(\bm y^{\ast}, \bm x^{\ast}, \bm w^{\ast}, \bm \lambda^{\ast})$ is a stationary point of $\mathcal{L}_{\beta}$. That is, $0 \in \partial \mathcal{L}_{\beta}(\bm y^{\ast}, \bm x^{\ast}, \bm w^{\ast}, \bm \lambda^{\ast})$.
\end{theorem}
To prove Theorem \ref{convergence theorem}, we define the following functions:
\begin{equation*}
F(\bm y)=\sum_{i=1}^{n^2}\|\mathbf y_{i}\|_{2},
\end{equation*}
\begin{equation*}
G(\bm x, \bm w)=\frac{\mu}{2} \| A(\bm w) \bm x - \bm d \|_2^2 +S(\bm w)+\delta_{\mathcal{C}_{\bm x}}(\bm x)+\delta_{{\mathcal{C}}_{\bm  w}}(\bm w),
\end{equation*}
\begin{equation*}
g(\bm x, \bm w)=\frac{\mu}{2} \| A(\bm w) \bm x - \bm d \|_2^2+S(\bm w).
\end{equation*}
\begin{lemma}
\label{lemma1}
The iterates of Algorithm 1 satisfy:\\
1. $\nabla_{\bm x}g(\bm x^{k+1}, \bm w^{k+1})=-\sum_{i=1}^{n^2}D_{i}^{T}{\bm \lambda_{i}^{k+1}}+\eta^{k+1}_{\bm x}$.\\
2. $\left\|\sum_{i=1}^{n^2}D_{i}^{T}{\bm \lambda_{i}^{k+1}}-\sum_{i=1}^{n^2}D_{i}^{T}{\bm \lambda_{i}^{k}}\right\|_{2}\leq C_{0}\left\|\bigg[\begin{array}{c}{\bm x^{k+1}-\bm x^{k}} \\ {\bm w^{k+1}-\bm w^{k}}\end{array}\bigg]\right\|_{2}+\|\eta^{k+1}_{\bm x}-\eta^{k}_{\bm x}\|_{2}$,
where $C_{0}$ is a constant.
\end{lemma}
\begin{proof}
First, notice that for $(\bm x^{k+1}, \bm w^{k+1})$ generated by ADMM, we have $\bm x^{k+1}= [x^{k+1}_{1}; x^{k+1}_{2}; \dots; x^{k+1}_{n^2}] \in {\mathcal{C}_{\bm x}}$, $\bm w^{k+1}= [w^{k+1}_{1}; w^{k+1}_{2}; \dots; w^{k+1}_{p}]  \in {{\mathcal{C}}_{\bm w}}$. From the definition of the general subgradient, for all $\bm v_{\bm x} \in \partial \delta_{{\mathcal{C}_{\bm x}}}(\bm x^{k+1})$,
\begin{equation*}
\delta_{{\mathcal{C}_{\bm x}}}(\bm x)-\delta_{{\mathcal{C}_{\bm x}}}(\bm x^{k+1})\geq \bm v_{\bm x}\cdot (\bm x-\bm x^{k+1}), \ \ \ \forall \bm x \in \mathbb{R}^{n^2}.
\end{equation*}
For any $\bm x= [x_{1}; x_{2}; \dots; x_{n^2}]$, where $x_{i}\geq x_{i}^{k+1}\geq 0, i=1,2,\dots, n^2$, we have
$\bm 0 \geq \bm v_{\bm x}\cdot(\bm x-\bm x^{k+1})$.
For any $\bm x= [x_{1}; x_{2}; \dots;  x_{n^2}]$, where $0\leq x_{i}\leq x_{i}^{k+1}, i=1,2,\dots, n^2$, we have $\bm 0 \leq \bm v_{\bm x}\cdot(\bm x-\bm x^{k+1}).$
Thus, we have $\bm v_{\bm x} = \bm 0$. Similarly, $\bm v_{\bm w} = \bm 0$.
By the first-order optimality condition at $(\bm x^{k+1}, \bm w^{k+1})$,
\begin{equation*}
\begin{aligned}
\begin{bmatrix}
\nabla_{\bm x} g(\bm x^{k+1}, \bm w^{k+1})\\
\nabla_{\bm w}g (\bm x^{k+1}, \bm w^{k+1})
\end{bmatrix}
+
\begin{bmatrix}
\sum_{i=1}^{n^2}D_{i}^{T}{\bm \lambda_{i}^{k}}\\
\bm 0
\end{bmatrix}
-
\begin{bmatrix}
\sum_{i=1}^{n^2}\beta D_{i}^{T}(\bm y_{i}^{k+1}-D_{i} \bm x)   \\
\bm 0
\end{bmatrix}
-
\begin{bmatrix}
\eta^{k+1}_{\bm x}   \\
\eta^{k+1}_{\bm w}
\end{bmatrix}
\in \partial \delta_{{\mathcal{C}_{\bm x}}\times {{\mathcal{C}}_{\bm w}}}(\bm x^{k+1}, \bm w^{k+1}).
\end{aligned}
\end{equation*}
Hence,
\begin{equation*}
{\nabla_{\bm x}} g(\bm x^{k+1}, \bm w^{k+1})+\sum_{i=1}^{n^2}\left(D_{i}^{T}{\bm \lambda_{i}^{k}}-\beta D_{i}^{T}(\bm y_{i}^{k+1}-D_{i} \bm x)\right)-\eta^{k+1}_{\bm x} = \bm 0.
\end{equation*}
Then by the update of the multiplier $\bm \lambda^{k+1}=\bm \lambda^{k}-\beta \left(\bm y^{k+1}-D\bm x^{k+1}\right)$, we obtain
\begin{equation}
\label{gradient g}
\nabla_{\bm x} g(\bm x^{k+1}, \bm w^{k+1})=-\sum_{i=1}^{n^2}D_{i}^{T}{\bm \lambda_{i}^{k+1}}+\eta^{k+1}_{\bm x}.
\end{equation}
We make use of the decomposition
\begin{equation}
\begin{aligned}
\label{decomposition}
&\|\nabla_{\bm x} g(\bm x^{k}, \bm w^{k})-\nabla_{\bm x} g(\bm x^{k+1}, \bm w^{k+1})\|_{2}\\
\leq &\|\nabla_{\bm x} g(\bm x^{k}, \bm w^{k})-\nabla_{\bm x} g(\bm x^{k}, \bm w^{k+1})\|_{2}+\|\nabla_{\bm x} g(\bm x^{k}, \bm w^{k+1})-\nabla_{\bm x} g(\bm x^{k+1}, \bm w^{k+1})\|_{2}.\\
\end{aligned}
\end{equation}
For the term $\|\nabla_{\bm x} g(\bm x^{k}, \bm w^{k})-\nabla_{\bm x} g(\bm x^{k}, \bm w^{k+1})\|_{2}$, we have the estimate
\begin{equation}
\begin{aligned}
\label{decomposition1}
&\|\nabla_{\bm x} g(\bm x^{k}, \bm w^{k})-\nabla_{\bm x} g(\bm x^{k}, \bm w^{k+1})\|_{2}\\
=& \|A(\bm w^{k})^{T}A(\bm w^{k})\bm x^{k}-A(\bm w^{k})^{T}\bm d-A(\bm w^{k+1})^{T}A(\bm w^{k+1})\bm x^{k}+A(\bm w^{k+1})^{T}\bm d\|_{2}\\
\leq &\|A(\bm w^{k})^{T}A(\bm w^{k})\bm x^{k}-A(\bm w^{k+1})^{T}A(\bm w^{k})\bm x^{k}\|_{2}+\|A(\bm w^{k+1})^{T}A(\bm w^{k})\bm x^{k}-A(\bm w^{k+1})^{T}A(\bm w^{k+1})\bm x^{k} \|_{2}\\
&+\|A(\bm w^{k})^{T}\bm d-A(\bm w^{k+1})^{T}\bm d\|_{2}.
\end{aligned}
\end{equation}
Then, we estimate the terms in (\ref{decomposition1}) separately with the 2-weight case, i.e., $\bm w = [w; 1-w]$. Cases with $p\geq 3$ can be considered similarly, here we omit it.
First,
\begin{equation}
\begin{aligned}
&\|A(\bm w^{k})^{T}A(\bm w^{k})\bm x^{k}-A(\bm w^{k+1})^{T}A(\bm w^{k})\bm x^{k}\|_{2}\\
\leq &\|A(\bm w^{k})^{T}-A(\bm w^{k+1})^{T}\|_{2}\cdot  \|A(\bm w^{k})\bm x^{k}\|_{2}\\
\leq &\|w^{k}A_{1}^{T}+(1-w^{k})A_{2}^{T}-w^{k+1}A_{1}^{T}-(1-w^{k+1})A_{2}^{T}\|_{2}\cdot  \|A(\bm w^{k})\bm x^{k}\|_{2}\\
\leq & \|A_{1}^{T}-A_{2}^{T}\|_{2}\cdot  \|A(\bm w^{k})\bm x^{k}\|_{2}\cdot\|w^{k}-w^{k+1}\|_{2},\\
= &\lambda_{\rm max}(A_{1}^{T}-A_{2}^{T})\cdot  \|A(\bm w^{k})\bm x^{k}\|_{2}\cdot\|w^{k}-w^{k+1}\|_{2},\\
= &c_{1}\|w^{k}-w^{k+1}\|_{2},
\end{aligned}
\end{equation}
where $c_{1}:=\lambda_{\rm max}(A_{1}^{T}-A_{2}^{T})\cdot \|A(\bm w^{k})\bm x^{k}\|_{2}$ is a constant.
Similarly,
\begin{equation}
\begin{aligned}
&\|A(\bm w^{k+1})^{T}A(\bm w^{k})\bm x^{k}-A(\bm w^{k+1})^{T}A(\bm w^{k+1})\bm x^{k} \|_{2}\\
\leq &\lambda_{\rm max}(A_{1}-A_{2})\cdot \|A(\bm w^{k+1})^{T}\|_{2}\cdot \|\bm x^{k}\|_{2}\cdot\|w^{k}-w^{k+1}\|_{2},\\
= &c_{2}\|w^{k}-w^{k+1}\|_{2},
\end{aligned}
\end{equation}
where $c_{2}:=\lambda_{\rm max}(A_{1}-A_{2})\cdot \|A(\bm w^{k+1})^{T}\|_{2}\cdot \|\bm x^{k}\|_{2}$ is a constant.

Moreover,
\begin{equation}
\begin{aligned}
&\|A(\bm w^{k})^{T}\bm d-A(\bm w^{k+1})^{T}\bm d\|_{2}\\
\leq &\|(A_{1}-A_{2})^{T}\bm d\|_{2}\cdot \|w^{k}-w^{k+1}\|_{2}\\
= &c_{3}\|w^{k}-w^{k+1}\|_{2},
\end{aligned}
\end{equation}
where $c_{3}:=\|(A_{1}-A_{2})^{T}\bm d\|_{2}$ is a constant.

For the term $\|\nabla_{\bm x} g(\bm x^{k}, \bm w^{k+1})-\nabla_{\bm x} g(\bm x^{k+1}, \bm w^{k+1})\|_{2}$, we have
\begin{equation}
\begin{aligned}
\label{decomposition2}
&\|\nabla_{\bm x} g(\bm x^{k}, \bm w^{k+1})-\nabla_{\bm x} g(\bm x^{k+1}, \bm w^{k+1})\|_{2}\\
= &\mu\|A(\bm w^{k+1})^{T}A(\bm w^{k+1})(\bm x^{k}-\bm x^{k+1})\|_{2} \\
\leq &\mu\lambda_{\rm max}(A(\bm w^{k+1})^{T}A(\bm w^{k+1}))\cdot\|\bm x^{k}-\bm x^{k+1}\|_{2},
\end{aligned}
\end{equation}
Then we know from (\ref{gradient g})-(\ref{decomposition2}) that there exists a constant $C_{0}$ such that
\begin{equation*}
\begin{aligned}
&\left\|\sum_{i=1}^{n^2}D_{i}^{T}{\bm \lambda_{i}^{k+1}}-\sum_{i=1}^{n^2}D_{i}^{T}{\bm \lambda_{i}^{k}}\right\|_{2}-\|\eta^{k}_{\bm x}-\eta^{k+1}_{\bm x} \|_{2}\\
\leq &\left\|-\sum_{i=1}^{n^2}D_{i}^{T}{\bm \lambda_{i}^{k+1}}+\eta^{k+1}_{\bm x} +\sum_{i=1}^{n^2}D_{i}^{T}{\bm \lambda_{i}^{k}}-\eta^{k}_{\bm x} \right\|_{2}\\
=&\|\nabla_{\bm x} g(\bm x^{k+1}, \bm w^{k+1})-\nabla_{\bm x} g(\bm x^{k}, \bm w^{k})\|_{2}\\
\leq &C_{0}\left\|\bigg[\begin{array}{c}{\bm x^{k+1}-\bm x^{k}} \\ {\bm w^{k+1}-\bm w^{k}}\end{array}\bigg]\right\|_{2},
\end{aligned}
\end{equation*}
where $C_{0}:=\sqrt{2}{\rm max}\{c_{1}+c_{2}+c_{3}, \mu \lambda_{\rm max}(A(\bm w^{k+1})^{T}A(\bm w^{k+1}))\}$.
Hence, we have
\begin{equation*}
\left\|\sum_{i=1}^{n^2}D_{i}^{T}{\bm \lambda_{i}^{k+1}}-\sum_{i=1}^{n^2}D_{i}^{T}{\bm \lambda_{i}^{k}}\right\|\leq C_{0}\left\|\bigg[\begin{array}{c}{\bm x^{k+1}-\bm x^{k}} \\ {\bm w^{k+1}-\bm w^{k}}\end{array}\bigg]\right\|_{2}^{2}+\|\eta^{k+1}_{\bm x}-\eta^{k}_{\bm x}\|_{2}.
\end{equation*}
\end{proof}

\begin{lemma}
\label{lemma2}
Let $\{(\bm y^k, \bm x^k, \bm w^k, \bm \lambda^k)\}$ be the sequence of iterates generated by Algorithm \ref{alg:ADMM-LAP}. If $\beta>\frac{1}{a_{1}}\lambda_{\rm max}(A(\bm w^{k})^{T}A(\bm w^{k}))$ and $\beta$ satisfies $\beta^{2}a_{1}-L\beta-4a_{2}C_{0}>0$, where $a_{1}$, $a_{2}$ are constants depending on $D$, $L$ denotes the Lipschitz constant of $-g(\bm x^{k+1}, \bm w^{k})$, then $\{(\bm y^k, \bm x^k, \bm w^k, \bm \lambda^k)\}$ satisfies:\\
1. $\mathcal{L}_{\beta}(\bm y^k, \bm x^k, \bm w^k, \bm \lambda^k)$ is lower bounded and there is a constant $C_{1}>0$ such that for all sufficiently large k, we have
\begin{equation*}
\begin{aligned}
&\mathcal{L}_{\beta}(\bm y^k, \bm x^k, \bm w^k, \bm \lambda^k)-\mathcal{L}_{\beta}(\bm y^{k+1}, \bm x^{k+1}, \bm w^{k+1}, \bm \lambda^{k+1}) \\
\geq &C_{1}\left(\sum_{i=1}^{n^2}\|\mathbf y_{i}^{k}-\mathbf y_{i}^{k+1}\|_{2}^{2}+\left\|\bigg[\begin{array}{c}{\bm x^{k+1}-\bm x^{k}} \\ {\bm w^{k+1}-\bm w^{k}}\end{array}\bigg]\right\|_{2}^{2}\right)- \frac{2}{\beta}\|\eta^{k+1}_{\bm x}-\eta^{k}_{\bm x}\|_{2}^{2}.
\end{aligned}
\end{equation*}
\\
2. $\{(\bm y^k, \bm x^k, \bm w^k, \bm \lambda^k)\}$ is bounded.
\end{lemma}
\begin{proof}
According to the optimality condition of the $\bm y$-subproblem, we define
\begin{equation*}
d^{k+1}_{i}:=\left(\bm \lambda^{k}\right)_{i}+\beta(\bm y^{k+1}_{i}-D_{i}\bm x^{k})\in \partial \|\mathbf y_{i}^{k+1}\|_{2}.
\end{equation*}
We know from the definition of $\mathcal{L}_{\beta}(\bm y^k, \bm x^k, \bm w^k, \bm \lambda^k)$ that
\begin{equation*}
\begin{aligned}
&\mathcal{L}_{\beta}(\bm y^k, \bm x^k, \bm w^k, \bm \lambda^k)-\mathcal{L}_{\beta}(\bm y^{k+1}, \bm x^k, \bm w^k, \bm \lambda^k)\\
=& F(\bm y^{k})-F(\bm y^{k+1})-\sum_{i=1}^{n^2}\left(\bm \lambda^{k}\right)_{i}^{T}(\mathbf y_{i}^{k}-\mathbf y_{i}^{k+1})+\frac{\beta}{2}\sum_{i=1}^{n^2}\|\mathbf y_{i}^{k}-D_{i}\bm x^{k}\|_{2}-\frac{\beta}{2}\sum_{i=1}^{n^2}\|\mathbf y_{i}^{k+1}-D_{i}\bm x^{k}\|_{2}\\
=&F(\bm y^{k})-F(\bm y^{k+1})-\sum_{i=1}^{n^2}\left(\bm \lambda^{k}\right)_{i}^{T}(\mathbf y_{i}^{k}-\mathbf y_{i}^{k+1})+\frac{\beta}{2}\sum_{i=1}^{n^2}\left(\|\mathbf y_{i}^{k}-\mathbf y_{i}^{k+1}\|_{2}^{2}+2\langle \mathbf y_{i}^{k+1}-D_{i}\bm x^{k}, \mathbf y_{i}^{k}-\mathbf y_{i}^{k+1}\rangle\right)\\
= &\sum_{i=1}^{n^2}\left(\|\mathbf y_{i}^{k}\|_{2}-\|\mathbf y_{i}^{k+1}\|_{2}-\left\langle d^{k+1}_{i}, \mathbf y_{i}^{k}-\mathbf y_{i}^{k+1}\right\rangle\right)+\frac{\beta}{2}\sum_{i=1}^{n^2}\|\mathbf y_{i}^{k}-\mathbf y_{i}^{k+1}\|_{2}^{2}\\
\geq &\frac{\beta}{2}\sum_{i=1}^{n^2}\|\mathbf y_{i}^{k}-\mathbf y_{i}^{k+1}\|_{2}^{2},
\end{aligned}
\end{equation*}
where the second equality follows from the cosine rule: $\|b+c\|^{2}-\|a+c\|^{2}=\|b-a\|^{2}+2\langle a+c, b-a\rangle$ and the last inequality follows from the convexity of $\|\mathbf y\|_{2}$.

Moreover, we have
\begin{equation*}
\begin{aligned}
&\mathcal{L}_{\beta}(\bm y^{k+1}, \bm x^k, \bm w^k, \bm \lambda^k)-\mathcal{L}_{\beta}(\bm y^{k+1}, \bm x^{k+1}, \bm w^{k+1}, \bm \lambda^{k+1})\\
=& g(\bm x^{k}, \bm w^{k})-g(\bm x^{k+1}, \bm w^{k+1})
-\sum_{i=1}^{n^2}\left((\bm \lambda^{k})_{i}^{T}-(\bm \lambda^{k+1})_{i}^{T}\right)\mathbf y^{k+1}_{i}+\sum_{i=1}^{n^2}\left(\bm \lambda^{k}\right)_{i}^{T}D_{i}\bm x^{k}-\sum_{i=1}^{n^2}\left(\bm \lambda^{k+1}\right)_{i}^{T}D_{i}\bm x^{k+1}\\
&+\sum_{i=1}^{n^2}\frac{\beta}{2}\|\mathbf y_{i}^{k+1}-D_{i}\bm x^{k}\|_{2}+\sum_{i=1}^{n^2}\frac{\beta}{2}\|\mathbf y_{i}^{k+1}-D_{i}\bm x^{k+1}\|_{2}\\
=& g(\bm x^{k}, \bm w^{k})-g(\bm x^{k}, \bm w^{k+1})+g(\bm x^{k}, \bm w^{k+1})-g(\bm x^{k+1}, \bm w^{k+1})-\sum_{i=1}^{n^2}\left((\bm \lambda^{k})_{i}^{T}-(\bm \lambda^{k+1})_{i}^{T}\right)\mathbf y^{k+1}_{i}\\
&+\sum_{i=1}^{n^2}\left(\bm \lambda^{k}\right)_{i}^{T}D_{i}\bm x^{k}-\sum_{i=1}^{n^2}\left(\bm \lambda^{k+1}\right)_{i}^{T}D_{i}\bm x^{k}+\sum_{i=1}^{n^2}\left(\bm \lambda^{k+1}\right)_{i}^{T}D_{i}\bm x^{k}-\sum_{i=1}^{n^2}\left(\bm \lambda^{k+1}\right)_{i}^{T}D_{i}\bm x^{k}\\
&+\frac{\beta}{2}\sum_{i=1}^{n^2}\left(\|-D_{i}\bm x^{k}+D_{i}\bm x^{k+1}\|_{2}+2\langle -D_{i}\bm x^{k+1}+\mathbf y_{i}^{k+1}, -D_{i}\bm x^{k}+D_{i}\bm x^{k+1}\rangle\right)\\
 \geq & -\frac{L}{2}\|\bm w^{k+1}-\bm w^{k}\|_{2}^{2} +\frac{\beta}{2}\sum_{i=1}^{n^2}\left(\|-D_{i}\bm x^{k}+D_{i}\bm x^{k+1}\|_{2}+2 \langle-D_{i}\bm x^{k+1}+\mathbf y_{i}^{k+1}, -D_{i}\bm x^{k}+D_{i}\bm x^{k+1}\rangle\right)\\
&+\sum_{i=1}^{n^2}\beta \langle\mathbf y_{i}^{k+1}-D_{i}\bm x^{k+1}, -\mathbf y_{i}^{k+1}+D_{i}\bm x^{k+1} \rangle\\
=&-\frac{L}{2}\|\bm w^{k+1}-\bm w^{k}\|_{2}^{2}-\frac{1}{\beta}\sum_{i=1}^{n^2}\|\left(\bm \lambda^{k+1}\right)_{i}-\left(\bm \lambda^{k}\right)_{i}\|_{2}^{2}+\frac{\beta}{2}\sum_{i=1}^{n^2}\|-D_{i}\bm x^{k}+D_{i}\bm x^{k+1}\|_{2}^{2}\\
\geq& -\frac{L}{2}\left\|\bigg[\begin{array}{c}{\bm x^{k+1}-\bm x^{k}} \\ {\bm w^{k+1}-\bm w^{k}}\end{array}\bigg]\right\|_{2}^{2}-\frac{2a_{2}C_{0}}{\beta}\left\|\bigg[\begin{array}{c}{\bm x^{k+1}-\bm x^{k}} \\ {\bm w^{k+1}-\bm w^{k}}\end{array}\bigg]\right\|_{2}^{2}+\frac{\beta a_{1}}{2}\left\|\bigg[\begin{array}{c}{\bm x^{k+1}-\bm x^{k}} \\ {\bm w^{k+1}-\bm w^{k}}\end{array}\bigg]\right\|_{2}^{2}-\frac{2}{\beta}\|\eta^{k}_{\bm x}-\eta^{k+1}_{\bm x}\|_{2}^{2}\\
=& C\left\|\bigg[\begin{array}{c}{\bm x^{k+1}-\bm x^{k}} \\ {\bm w^{k+1}-\bm w^{k}}\end{array}\bigg]\right\|_{2}^{2}-\frac{2}{\beta}\|\eta^{k}_{\bm x}-\eta^{k+1}_{\bm x}\|_{2}^{2},
\end{aligned}
\end{equation*}
where $C:=-\frac{L}{2}-\frac{2a_{2}C_{0}}{\beta}+\frac{\beta a_{1}}{2}$ is a constant, $a_{1}, a_{2}$ are constants such that for $i=1, 2, \dots, n^{2}$, $\|D_{i}(\bm x^{k+1}-\bm x^{k})\|_{2}^{2}\geq a_{1}\|\bm x^{k+1}-\bm x^{k}\|_{2}^{2}$ and $\|\bm (\lambda^{k+1})_{i}-(\lambda^{k})_{i}\|_{2}^{2}\leq a_{2}\|D_{i}^{T}\left(\bm (\lambda^{k+1})_{i}-(\lambda^{k})_{i}\right)\|_{2}^{2}$, respectively. The second equality follows from the cosine rule; the third inequality follows from the convexity of $g(\bm x^{k}, \bm w^{k+1})$ with respect to $\bm x$ and the Lipschitz differentiable property of $-g(\bm x^{k+1}, \bm w^{k})$ with respect to $\bm w$ \cite{Wang2019}, $L$ is the Lipschitz constant; the fourth equality follows from the definition of the multiplier $\bm \lambda^{k+1}=\bm \lambda^{k}-\beta \left(\bm y^{k+1}-D\bm x^{k+1}\right)$; the fifth inequality follows from Lemma \ref{lemma1}. In order to ensure $C>0$, we require $\beta$ satisfies $\beta^{2}a_{1}-L\beta-4a_{2}C_{0}>0$.

Then we know from two equalities above that
\begin{equation}
\begin{aligned}
&\mathcal{L}_{\beta}(\bm y^k, \bm x^k, \bm w^k, \bm \lambda^k)-\mathcal{L}_{\beta}(\bm y^{k+1}, \bm x^{k+1}, \bm w^{k+1}, \bm \lambda^{k+1}) \\
\geq &C_{1}\left(\sum_{i=1}^{n^2}\|\mathbf y_{i}^{k}-\mathbf y_{i}^{k+1}\|_{2}^{2}+\left\|\bigg[\begin{array}{c}{\bm x^{k+1}-\bm x^{k}} \\ {\bm w^{k+1}-\bm w^{k}}\end{array}\bigg]\right\|_{2}^{2}\right)- \frac{2}{\beta}\|\eta^{k}_{\bm x}-\eta^{k+1}_{\bm x}\|_{2}^{2}.
\end{aligned}
\end{equation}
This means for all sufficiently large $k$, $\mathcal{L}_{\beta}(\bm y^k, \bm x^k, \bm w^k, \bm \lambda^k)$ is nonincreasing. As $D_{i}\in \mathbb{R}^{2\times n^2}$ has full row rank, then there exists at one $\hat {\bm x}$ such that $D_{i}\hat {\bm x} = {\mathbf y}_{i}^{k}, \ \forall i=1,2,...,n^2$.
Thus, we arrive at
\begin{equation*}
\begin{aligned}
&\mathcal{L}_{\beta}(\bm y^k, \bm x^k, \bm w^k, \bm \lambda^k)\\
=& \frac{\mu}{2} \| A(\bm w^{k}) \bm x^{k} - \bm d \|_2^2+\sum_{i=1}^{n^2}\|\mathbf y_{i}^{k}\|_{2}-\sum_{i=1}^{n^2}\left(\bm \lambda^{k}\right)_{i}^{T}(\mathbf y_{i}^{k}-D_{i}\bm x^{k})+\frac{\beta}{2}\sum_{i=1}^{n^2}\|\mathbf y_{i}^{k}-D_{i}\bm x^{k}\|_{2}^2\\
=&\sum_{i=1}^{n^2}\|\mathbf y_{i}^{k}\|_{2}+\frac{\beta}{2}\sum_{i=1}^{n^2}\|\mathbf y_{i}^{k}-D_{i}\bm x^{k}\|_{2}^2+\frac{\mu}{2} \| A(\bm w^{k}) \bm x^{k} - \bm d \|_2^2-\sum_{i=1}^{n^2}\langle D_{i}^{T}\left(\bm \lambda^{k}\right)_{i}, \hat{\bm x}-\bm x^{k}\rangle,\\
=&\sum_{i=1}^{n^2}\|\mathbf y_{i}^{k}\|_{2}+\frac{\beta}{2}\sum_{i=1}^{n^2}\|\mathbf y_{i}^{k}-D_{i}\bm x^{k}\|_{2}^2+\frac{\mu}{2} \| A(\bm w^{k}) \hat{\bm x}^{k} - \bm d \|_2^2\\
=&\sum_{i=1}^{n^2}\|\mathbf y_{i}^{k}\|_{2}+\frac{\mu}{2} \| A(\bm w^{k}) \hat{\bm x}^{k} - \bm d \|_2^2+\frac{\beta}{2}\sum_{i=1}^{n^2}\|D_{i}(\hat{\bm x}-\bm x^{k})\|_{2}^{2}\\
\geq &\sum_{i=1}^{n^2}\|\mathbf y_{i}^{k}\|_{2}+\frac{\mu}{2} \| A(\bm w^{k}) \hat{\bm x}^{k} - \bm d \|_2^2+\frac{\beta a_{1}}{2}\|\hat{\bm x}-\bm x^{k}\|_{2}^{2}\\
>&-\infty.
\end{aligned}
\end{equation*}
Since $\mathcal{L}_{\beta}(\bm y^k, \bm x^k, \bm w^k, \bm \lambda^k)$ is lower bounded and $\sum_{i=1}^{n^2}\|\mathbf y_{i}^{k}\|_{2}+\frac{\mu}{2} \| A(\bm w^{k}) \hat{\bm x}^{k}\|_{2}^{2}+(\frac{1}{2}\beta a_{1}-\frac{1}{2}\lambda_{\rm max}(A(\bm w^{k})^{T}A(\bm w^{k})))\|\hat{\bm x}-\bm x^{k}\|_{2}^{2}$ is coercive over the feasible set $\Omega_{F}:=\{(\bm y, \bm x, \bm w): D_{i}\bm x= \mathbf y_{i}, \ i=1,2,\dots,n^{2}\}$, we conclude that $\{\bm y^k\}$ and $\{(\bm x^k, \bm w^k)\}$ are bounded.
Hence, by Lemma \ref{lemma1}, $\{\bm \lambda^k\}$ is bounded.
\end{proof}

\begin{lemma}
\label{lemma3}
Let $\partial \mathcal{L}_{\beta}(\bm y^{k+1}, \bm x^{k+1}, \bm w^{k+1}, \bm \lambda^{k+1})=\left(\partial_{\bm y} \mathcal{L}_{\beta}, \nabla_{(\bm x, \bm w)} \mathcal{L}_{\beta}, \nabla_{\bm \lambda} \mathcal{L}_{\beta}\right)$. Then there exists a constant $\tilde C>0$ such that for all $k\geq 1$, for some $p^{k+1}\in \partial \mathcal{L}_{\beta}(\bm y^{k+1}, \bm x^{k+1}, \bm w^{k+1}, \bm \lambda^{k+1})$, we have
\begin{equation*}
\|p^{k+1}\|_{2}\leq \tilde C\left(\left\|\bigg[\begin{array}{c}{\bm x^{k+1}-\bm x^{k}} \\ {\bm w^{k+1}-\bm w^{k}}\end{array}\bigg]\right\|_{2}+\|\eta^{k+1}_{\bm x}-\eta^{k}_{\bm x}\|_{2}\right).
\end{equation*}
\end{lemma}
\begin{proof}
Because
\begin{equation*}
\begin{aligned}
\nabla_{\bm \lambda_{i}} L_{\beta}(\bm y^{k+1}, \bm x^{k+1}, \bm w^{k+1}, \bm \lambda^{k+1})=-(\mathbf y_{i}^{k+1}-D_{i}\bm x^{k+1})=\frac{1}{\beta}\left(\left(\bm \lambda^{k+1}\right)_{i}-\left(\bm \lambda^{k}\right)_{i}\right),
\end{aligned}
\end{equation*}
and
\begin{equation*}
\nabla_{(\bm x, \bm w)} L_{\beta}(\bm y^{k+1}, \bm x^{k+1}, \bm w^{k+1}, \bm \lambda^{k+1})=
\begin{bmatrix}
\sum_{i=1}^{n^2}D_{i}^{T}\left(\left( \bm \lambda^{k+1}\right)_{i}-\left(\bm \lambda^{k}\right)_{i}\right)\\
\bm 0
\end{bmatrix},
\end{equation*}
we have,
\begin{equation*}
\begin{aligned}
\|\nabla_{\bm \lambda} L_{\beta}(\bm y^{k+1}, \bm x^{k+1}, \bm w^{k+1}, \bm \lambda^{k+1})\|_{2}=&\frac{1}{\beta}\sum_{i=1}^{n^2}\|\left(\bm \lambda^{k+1}\right)_{i}-\left(\bm \lambda^{k}\right)_{i}\|_{2}\\
\leq &\frac{1}{\beta\sqrt{a_{1}}}\sum_{i=1}^{n^2}\|D_{i}^{T}\left(\left(\bm \lambda^{k+1}\right)_{i}-\left(\bm \lambda^{k}\right)_{i}\right)\|_{2}\\
\leq &\frac{1}{\beta\sqrt{a_{1}}}
\left(C_{0}\left\|\bigg[\begin{array}{c}{\bm x^{k+1}-\bm x^{k}} \\ {\bm w^{k+1}-\bm w^{k}}\end{array}\bigg]\right\|_{2}+ \|\eta^{k+1}_{\bm x}-\eta^{k}_{\bm x}\|_{2}\right)
\end{aligned}
\end{equation*}
and
\begin{equation*}
\|\nabla_{(\bm x, \bm w)} L_{\beta}(\bm y^{k+1}, \bm x^{k+1}, \bm w^{k+1}, \bm \lambda^{k+1})\|_{2} \leq C_{0}\left\|\bigg[\begin{array}{c}{\bm x^{k+1}-\bm x^{k}} \\ {\bm w^{k+1}-\bm w^{k}}\end{array}\bigg]\right\|_{2}+ \|\eta^{k+1}_{\bm x}-\eta^{k}_{\bm x}\|_{2}.
\end{equation*}
By the definition of the subgradient, we make use of the decomposition and obtain
\begin{equation*}
\begin{aligned}
&\partial_{\bm y}L_{\beta}(\bm y^{k+1}, \bm x^{k+1}, \bm w^{k+1}, \bm \lambda^{k+1})\\
=&\partial \sum_{i=1}^{n^2}\left(\|\mathbf y_{i}^{k+1}\|_{2}-\left(\bm \lambda^{k+1}\right)_{i}+\beta(\mathbf y_{i}^{k+1}-D_{i}\bm x^{k+1})\right)\\
=&\partial \sum_{i=1}^{n^2}\left(\|\mathbf y_{i}^{k+1}\|_{2}-\left(\bm \lambda^{k}\right)_{i}+\beta(\mathbf y_{i}^{k+1}-D_{i}\bm x^{k})-\left(\bm \lambda^{k+1}\right)_{i} +\left(\bm \lambda^{k}\right)_{i}+\beta(D_{i}\bm x^{k}-D_{i}\bm x^{k+1})\right).
\end{aligned}
\end{equation*}
Thus, according to the optimal condition
\begin{equation*}
0\in \partial \|\mathbf y_{i}^{k+1}\|_{2}-\left(\bm \lambda^{k}\right)_{i}+\beta(D_{i}\bm x^{k}-D_{i}\bm x^{k+1}),
\end{equation*}
we have,
\begin{equation*}
-\sum_{i=1}^{n^2}\left(\left(\bm \lambda^{k+1}\right)_{i}-\left(\bm \lambda^{k}\right)_{i}\right)+\beta(D_{i}\bm x^{k}-D_{i}\bm x^{k+1})\in \partial _{\bm y}L_{\beta}(\bm y^{k+1}, \bm x^{k+1}, \bm w^{k+1}, \bm \lambda^{k+1}).
\end{equation*}
Letting
\begin{equation*}
\begin{aligned}
p^{k+1}:=(p_{\bm y}^{k+1}, \ p_{(\bm x, \bm w)}^{k+1},\ p_{\bm \lambda}^{k+1}),
\end{aligned}
\end{equation*}
where
\begin{equation*}
\begin{aligned}
&p_{\bm y}^{k+1}:=-\sum_{i=1}^{n^2}\left(\left(\bm \lambda^{k+1}\right)_{i}-\left(\bm \lambda^{k}\right)_{i}\right)+\beta(D_{i}\bm x^{k}-D_{i}\bm x^{k+1}),\\
&p_{(\bm x, \bm w)}^{k+1}:=\sum_{i=1}^{n^2}D_{i}^{T}\left(\left(\bm \lambda^{k+1}\right)_{i}-\left(\bm \lambda^{k}\right)_{i}\right),\\
&p_{\bm \lambda}^{k+1}:=\frac{1}{\beta}\sum_{i=1}^{n^2}\left(\left(\bm \lambda^{k+1}\right)_{i}-\left(\bm \lambda^{k}\right)_{i}\right).
\end{aligned}
\end{equation*}
Then we have,
\begin{equation*}
\begin{aligned}
\|p^{k+1}\|_{2}\leq& \frac{C_{0}}{\beta\sqrt{a_{1}}}\left\|\bigg[\begin{array}{c}{\bm x^{k+1}-\bm x^{k}} \\ {\bm w^{k+1}-\bm w^{k}}\end{array}\bigg]\right\|_{2}+\frac{C_{0}}{\sqrt{a_{1}}}\left\|\bigg[\begin{array}{c}{\bm x^{k+1}-\bm x^{k}} \\ {\bm w^{k+1}-\bm w^{k}}\end{array}\bigg]\right\|_{2}+C_{0}\left\|\bigg[\begin{array}{c}{\bm x^{k+1}-\bm x^{k}} \\ {\bm w^{k+1}-\bm w^{k}}\end{array}\bigg]\right\|_{2}\\
&+\beta\sum_{i=1}^{n^2}\|D_{i}\|_{2}\|\bm x^{k+1}-\bm x^{k}\|_{2}+\left(\frac{C_{0}}{\beta\sqrt{a_{1}}}+\frac{C_{0}}{\sqrt{a_{1}}}+C_{0}\right)\|\eta^{k+1}_{\bm x}-\eta^{k}_{\bm x}\|_{2}\\
\leq& \tilde C \left(\left\|\bigg[\begin{array}{c}{\bm x^{k+1}-\bm x^{k}} \\ {\bm w^{k+1}-\bm w^{k}}\end{array}\bigg]\right\|_{2}+\|\eta^{k+1}_{\bm x}-\eta^{k}_{\bm x}\|_{2}\right),
\end{aligned}
\end{equation*}
where $\tilde C:=C_{0}(1+\frac{1}{\beta\sqrt{a_{1}}}\frac{1}{\sqrt{a_{1}}})+\beta\sum_{i=1}^{n^2}\|D_{i}\|_{2}$ is a constant.
\end{proof}

Now we give the proof to Theorem \ref{convergence theorem}.
\begin{proof}[Proof of Theorem \ref{convergence theorem}]
By Lemma \ref{lemma2}, the sequence $\{(\bm y^k, \bm x^k, \bm w^k, \bm \lambda^k)\}$ is bounded, so there exists a convergent subsequence $\{(\bm y^{n_{k}}, \bm x^{n_{k}}, \bm w^{n_{k}}, \bm \lambda^{n_{k}})\}$, i.e., $\{(\bm y^{n_{k}}, \bm x^{n_{k}}, \bm w^{n_{k}}, \bm \lambda^{n_{k}})\}$ converges to $(\bm y^{\ast}, \bm x^{\ast}, \bm w^{\ast}, \bm \lambda^{\ast})$ as $k \rightarrow \infty$.
Notice that the residual is summable, then we know from $\|\eta^{k}_{\bm x}-\eta^{k+1}_{\bm x}\|_{2}\leq \|\eta^{k}_{\bm x}\|_{2}+\|\eta^{k+1}_{\bm x}\|_{2}$ that when $k\rightarrow \infty$, $\|\eta^{k}_{\bm x}-\eta^{k+1}_{\bm x}\|_{2}\rightarrow 0$. Hence, for sufficiently large k, we have
$L_{\beta}(\bm y^{k}, \bm x^{k}, \bm w^{k}, \bm \lambda^{k})$ is nonincreasing and lower-bounded. We derive that when $k\rightarrow \infty$,\begin{equation*}
\mathcal{L}_{\beta}(\bm y^k, \bm x^k, \bm w^k, \bm \lambda^k)-\mathcal{L}_{\beta}(\bm y^{k+1}, \bm x^{k+1}, \bm w^{k+1}, \bm \lambda^{k+1})\rightarrow 0.
\end{equation*}
Then we know from the proof of Lemma \ref{lemma2} that when $k\rightarrow \infty$,
\begin{equation*}
\begin{aligned}
&\sum_{i=1}^{n^2}\|\mathbf y_{i}^{k}-\mathbf y_{i}^{k+1}\|_{2}^{2}\rightarrow 0,\\
&\left\|\bigg[\begin{array}{c}{\bm x^{k+1}-\bm x^{k}} \\ {\bm w^{k+1}-\bm w^{k}}\end{array}\bigg]\right\|_{2}^{2}\rightarrow 0.
\end{aligned}
\end{equation*}
Then we know from Lemma \ref{lemma3} that there exists $p^{k+1}\in \partial \mathcal{L}_{\beta}(\bm y^{k+1}, \bm x^{k+1}, \bm w^{k+1}, \bm \lambda^{k+1})$ such that when $k\rightarrow \infty$,
$\|p^{k+1}\|_{2}\rightarrow 0$. This leads to $\|p^{n_{k}}\|_{2}\rightarrow 0$ as $k\rightarrow \infty$.
Based on the definition of the general subgradient in \cite{Rockafellar2009}, we obtain that $0\in \partial \mathcal{L}_{\beta}(\bm y^{\ast}, \bm x^{\ast}, \bm w^{\ast}, \bm \lambda^{\ast})$, i.e., $(\bm y^{\ast}, \bm x^{\ast}, \bm w^{\ast}, \bm \lambda^{\ast})$ is a stationary point.
\end{proof}

\subsection{Complexity analysis}
Now we consider the computational complexity of Algorithm \ref{alg:ADMM-LAP} for (\ref{eq:objFctn*}). In Step 1, the main computational costs come from computing $D_{i}\bm x^{k}$ and $\frac{1}{\beta}\left(\bm \lambda^{k}\right)_{i}$. Since $D_{i}$ denotes the first-order finite difference of $\bm x$ at the $i$th pixel, computing $D_{i}\bm x^{k},\ i=1,\dots,n^2$ requires $2n^{2}$ flops. Computing the scalar-vector product $\frac{1}{\beta}\left(\bm \lambda^{k}\right)_{i},\ i=1,\dots,n^2$ requires $2n^{2}$ flops. Finally, computing the 2-norm of the vectors $D_{i}\bm x^{k}+\frac{1}{\beta}\left(\bm \lambda^{k}\right)_{i}\in \mathbb{R}^2,\ i=1,\dots,n^2$ needs  $6n^{2}$ flops. Thus, Step 1 costs $O\left(n^{2}\right)$.

In Step 2, we first consider the computational cost of \eqref{eq:rdp1*}. Suppose we use the conjugate gradient method to solve for $\delta \mathbf{x}_{\mathcal{I}}$ and set the maximum iteration number to a small constant. Then the cost is determined by the cost of multiplying $\mathbf {\hat J}_{\bm x}^{\top}(\bm I - \mathbf {\hat J}_{\bm w} (\mathbf {\hat J}_{\bm w}^{\top} \mathbf {\hat J}_{\bm w} + \nabla^2_{\bm w} {\hat S}(\bm w))^{-1} \mathbf {\hat J}_{\bm w}^{\top}) \mathbf {\hat J}_{\bm x} + \frac{1}{\mu} \nabla^2_{\bm x} {\hat R}\left(\bm x, \bm y^{k+1}, \bm \lambda^k\right)$ on a vector and computing
$-\mathbf {\hat J}_{\bm x}^{\top}(\bm I - \mathbf {\hat J}_{\bm w} (\mathbf {\hat J}_{\bm w}^{\top} \mathbf {\hat J}_{\bm w} + \nabla^2_{\bm w} {\hat S}(\bm w))^{-1} \mathbf {\hat J}_{\bm w}^{\top}) \bm r$ and $\mathbf {\hat J}_{\bm x}^{\top} \mathbf {\hat J}_{\bm w} (\mathbf {\hat J}_{\bm w}^{\top} \mathbf {\hat J}_{\bm w} + \nabla^2_{\bm w} {\hat S}(\bm w))^{-1} \nabla_{\bm w} {\hat S}(\bm w)$. Since ${\mathbf{J}}_{\bm x} = A\left(\bm w\right)^\top$ and each PSF matrix $A_i$ allows the use of fast Fourier transforms (FFTs) \cite{Xi2014, Xia2012} to multiply it with a vector, the cost of multiplying $\hat{\mathbf{J}}_{\bm x}$ on a vector is $O\left(n^2 \log n\right)$. Similarly, the cost of multiplying $\hat{\mathbf{J}}_{\bm w}$ on a vector is also $O\left(n^2 \log n\right)$ by FFTs. So the cost of the matrix-vector product of $\mathbf {\hat J}_{\bm x}^{\top}(\bm I - \mathbf {\hat J}_{\bm w} (\mathbf {\hat J}_{\bm w}^{\top} \mathbf {\hat J}_{\bm w} + \nabla^2_{\bm w} {\hat S}(\bm w))^{-1} \mathbf {\hat J}_{\bm w}^{\top}) \mathbf {\hat J}_{\bm x} $ and a vector is $O\left(n^2 \log n\right)$. Computing the matrix-vector product of $\nabla^2_{\bm x} {\hat R}\left(\bm x, \bm y^{k+1}, \bm \lambda^k\right)$ and a vector equals to computing the matrix-vector product of $\beta D^{T}D$ and a vector, and it can be done in $O\left(n^2\right)$. Moreover, $-\mathbf {\hat J}_{\bm x}^{\top}(\bm I - \mathbf {\hat J}_{\bm w} (\mathbf {\hat J}_{\bm w}^{\top} \mathbf {\hat J}_{\bm w} + \nabla^2_{\bm w} {\hat S}(\bm w))^{-1} \mathbf {\hat J}_{\bm w}^{\top}) \bm r$ and $\mathbf {\hat J}_{\bm x}^{\top} \mathbf {\hat J}_{\bm w} (\mathbf {\hat J}_{\bm w}^{\top} \mathbf {\hat J}_{\bm w} + \nabla^2_{\bm w} {\hat S}(\bm w))^{-1} \nabla_{\bm w} {\hat S}(\bm w)$ can be performed in $O\left(n^2 \log n\right)$ by FFTs. Thus the computational cost of (\ref{eq:rdp1*}) is $O\left(n^2 \log n\right)$. Next we consider the cost of \eqref{delta w}. The main computational costs in \eqref{delta w} are from computing the matrix-vector product ${\hat{\mathbf{J}}}_{\bm w}^{\top}\hat{\mathbf{J}}_{\bm x} \delta \bm x_{\mathcal{I}}$ and $\hat{\mathbf{J}}_{\bm w}\bm r$. Since these matrix-vector products can also be accelerated by FFTs, the computational cost of (\ref{delta w}) is $O\left(n^2 \log n\right)$. For (\ref{equ:activeset}), the main computational costs are from computing ${\tilde{\mathbf{J}}_{\bm x}^{\top} \bm r}$, ${\tilde{\mathbf{J}}_{\bm w}^{\top} \bm r}$. Similar to the analysis for (\ref{delta w}), the costs of computing ${\tilde{\mathbf{J}}_{\bm x}^{\top} \bm r}$ and ${\tilde{\mathbf{J}}_{\bm w}^{\top} \bm r}$ are $O\left(n^2 \log n\right)$ due to FFTs. The scalar-vector product $-\mu\left[\begin{array}{c}{\tilde{\mathbf{J}}_{\bm x}^{\top} \bm r} \\ {\tilde{\mathbf{J}}_{\bm w}^{\top} \bm r}\end{array}\right]$ can be done in $n^2+p$ flops.
In applications, we usually choose the image size to be $n\times n = 256\times 256$ and the number of the unknown weights parameterizing the blurring matrix is usually chosen to be $p=2, 3$ or some small values much smaller than $n^2$, so we always have $p \ll n^2$. Thus, the cost of (\ref{equ:activeset}) is $O\left(n^2 \log n\right)$.
In (\ref{equ:step}), computing the scalar-vector multiplication
$\gamma\left[\begin{array}{c}{\delta \bm x_{\mathcal{A}}} \\ {\delta \bm w_{\mathcal{A}}}\end{array}\right]$
requires $n^2+p$ flops and computing the summation of the vectors $\left[\begin{array}{c}{\delta \bm x_{\mathcal{I}}} \\ {\delta \bm w_{\mathcal{I}}}\end{array}\right]$ and $\gamma\left[\begin{array}{c}{\delta \bm x_{\mathcal{A}}} \\ {\delta \bm w_{\mathcal{A}}}\end{array}\right]$ also requires $n^2+p$ flops. Hence, the cost of (\ref{equ:step}) is $O(n^2)$.
The cost of (\ref{equ:armijo}) is mainly from forming $\nabla_{\bm x, \bm w} \hat{\Phi}\left(\bm x, \bm w, \bm y^{k+1}, \bm \lambda^{k}\right)$ and computing the inner product $\left(\nabla_{\bm x, \bm w} \hat{\Phi}\left(\bm x, \bm w, \bm y^{k+1}, \bm \lambda^{k}\right)\right)^{\top}\left[\begin{array}{c}{\delta \bm x} \\ {\delta \bm w}\end{array}\right]$. We know that $\nabla_{\bm x, \bm w} \hat{\Phi}\left(\bm x, \bm w, \bm y^{k+1}, \bm \lambda^{k}\right)=\left[\begin{array}{c}{A(\bm w)^{T}(A(\bm w)\bm x-\bm d)} \\ {{\mathbf{\hat J}}_{\bm w}^{T}(A(\bm w)\bm x-\bm d)}\end{array}\right]$, so the cost of forming $\nabla_{\bm x, \bm w} \hat{\Phi}\left(\bm x, \bm w, \bm y^{k+1}, \bm \lambda^{k}\right)$ is $O\left(n^2 \log n\right)$ by FFTs, and computing the inner product $\left(\nabla_{\bm x, \bm w} \hat{\Phi}\left(\bm x, \bm w, \bm y^{k+1}, \bm \lambda^{k}\right)\right)^{\top}\left[\begin{array}{c}{\delta \bm x} \\ {\delta \bm w}\end{array}\right]$ costs $O\left(n^2\right)$. Thus the computational cost of Step 2 is $O\left(n^2\log n\right)$.

In Step 4, the computational cost is dominated by the cost of computing $D\bm x^{k+1}$ and forming $\beta\bm y^{k+1}$ and $\beta D\bm x^{k+1}$. It is easy to see that these operations cost $O\left(n^{2}\right)$. Therefore, the total computational complexity of Algorithm \ref{alg:ADMM-LAP} is $O\left(n^{2}\log n\right)$.

\subsection{Comparison method}
In this section, we briefly discuss a benchmark method ADMM-BCD of ADMM-LAP used in the numerical comparisons in Section \ref{sec:num}.

BCD in ADMM-BCD stands for the block coordinate descent (BCD) method \cite{Nocedal1999, Xu2013}, which is another popular approach for solving coupled optimization problems. The main idea of BCD is partitioning the optimization variables into a number of blocks, then minimize the objective function cyclically over each block while fixing the remaining blocks at their last updated values until convergence. For AO retinal image problems we consider in this paper, the variables can be naturally separated into two subsets, one for the image variable $\bm x$ and another for the parameters $\bm w$. For the tightly coupled $\left(\bm x,\bm w\right)$-subproblem
\begin{equation*}
\underset{\bm x \in \mathcal{C}_{\bm x}, \bm w \in {{\mathcal{C}}}_{\bm w}}{\min}\frac{\mu}{2} \| A\left(\bm w\right) \bm x - \bm d \|_2^2+ S(\bm w)+R\left(\bm x, \bm y^{k+1},\bm \lambda^{k}\right),
\end{equation*}
given the initial point $\left(\bm x^{k},\bm w^{k}\right)$, the iterative format of BCD is as follows
\begin{align}
\label{xsubp}
&\bm x^{(l+1)}=\underset{\bm x \in \mathcal{C}_{\bm x}}{\operatorname{argmin}} \frac{\mu}{2} \| A\left(\bm w^{(l)}\right) \bm x - \bm d \|_2^2+ R\left(\bm x, \bm y^{k+1},\bm \lambda^{k}\right),\\
\label{wsubp}
&\bm w^{(l+1)}=\underset{\bm w \in {\mathcal{C}}_{\bm w}}{\operatorname{argmin}} \frac{\mu}{2} \| A\left(\bm w\right) \bm x^{(l+1)} - \bm d \|_2^2+S(\bm w),
\end{align}
where $l\in \mathbb{N}$ denotes the $l$-th iteration of BCD.

For the numerical experiments, we inexactly solve $(\ref{xsubp})$ and $(\ref{wsubp})$ by the projected Gauss-Newton method with bound constraints \cite{Haber2014}. The search directions $\delta \bm x_{\mathcal{I}}$ at $\bm x^{(l)}$ can be computed as follows:
\begin{equation}
\label{xdirection}
\left(\mu{\hat{\mathbf{J}}}_{\bm x}^{\top} \hat{{\mathbf{J}}}_{\bm x} + \nabla^2_{\bm x} \hat{R}\left(\bm x^{(l)}, \bm y ^{k+1},\bm \lambda ^{k}\right)\right) \delta \bm x_{\mathcal{I}}
= -\left(\mu\hat{{\mathbf{J}}}_{\bm x}^{\top} \bm r(\bm x^{(l)},\bm w^{(l)}) +\nabla_{\bm x} \hat{R}\left(\bm x^{(l)}, \bm y ^{k+1},\bm \lambda ^{k}\right)\right),
\end{equation}
where $\hat{\mathbf{J}}_{\bm x}$, $\nabla_{\bm x} \hat{R}$ and $\nabla^2_{\bm x} {\hat R}$ represent ${\mathbf{J}}_{\bm x}$, $\nabla_{\bm x} R$ and $\nabla^2_{\bm x} {R}$ restricted to the inactive set via projection, respectively. Moreover, $\delta \bm x_{\mathcal{A}}$ is given by the projected gradient descent step and $\delta \bm x$ can be computed by a scaled combination \cite{Haber2014}. Finally, the projected Armijo line search is applied to find the solution $\bm x^{(l+1)}$.  Then the search directions $\delta \bm w_{\mathcal{I}}$ at $\bm w^{(l)}$ can be computed as follows:
\begin{equation}
\label{wdirection}
\left(\hat{{\mathbf{J}}}_{\bm w}^{\top} \hat{{\mathbf{J}}}_{\bm w}+\nabla_{\bm w}^{2} S(\bm w)\right) \delta \bm w_{\mathcal{I}}=-\hat{{\mathbf{J}}}_{\bm w}^{\top}\bm r(\bm x^{(l+1)},\bm w^{(l)})+\nabla_{\bm w} S(\bm w),
\end{equation}
where $\hat{\mathbf{J}}_{\bm w}$, $\nabla_{\bm w} \hat S(\bm w)$ and $\nabla_{\bm w}^{2} \hat S(\bm w)$ represent ${\mathbf{J}}_{\bm w}$, $\nabla_{\bm w} S(\bm w)$ and $\nabla_{\bm w}^{2} S(\bm w)$ restricted to the inactive set via projection, respectively. Similar to the above discussion, we can obtain $\delta \bm w_{\mathcal{A}}$ by projected gradient descent and obtain $\delta \bm w$ by a scaled combination. Finally, the solution $\bm w^{(l+1)}$ can be obtained by the projected Armijo line search.

For $(\ref{xdirection})$, due to the large dimension, the search direction can be inexactly computed by the conjugate gradients method. For $(\ref{wdirection})$, because the dimension is small, the search direction can be computed by a direct solver.
Moreover, we note that BCD is fully decoupled while optimizing over one set of variables, and the optimization over another set of variables is neglected. This degrades the convergence for tightly coupled problems\cite{Nocedal1999}.

Similar to ADMM-LAP discussed in the previous section, ADMM-BCD uses ADMM as the outer optimization method to tackle the TV regularization and applies BCD to tackle the related $(\bm x, \bm w)$-subproblems appearing in each ADMM iteration. The main operations of ADMM-BCD are summarized in Algorithm \ref{alg:ADMM-BCD}.

\begin{algorithm}[H]
\caption{ADMM-BCD method for (\ref{eq:objFctn*})}
\label{alg:ADMM-BCD}
\normalsize
\begin{algorithmic}
\STATE{\textbf{Input:}} $\left(\bm y^{0},\bm x^{0},\bm w^{0},{\bm \lambda}^{0}\right)\in \mathbb{R}^{2n^2}\times \mathbb{R}^{n^2}\times \mathbb{R}^{p}\times\mathbb{R}^{2n^2}$, the penalty parameter $\beta>0$. Set $k=0$.\STATE{\textbf{Output:}} $\bm y^{k},\bm x^{k},\bm w^{k},{\bm \lambda}^{k}$.
\STATE{\textbf{Step 1}} Compute $\bm y^{k+1}$ using (\ref{ysubproblem}).
\STATE{\textbf{Step 2}} Compute $\bm x^{k+1}$, $\bm w^{k+1}$ by  iteratively solving (\ref{xsubp}) and (\ref{wsubp}).
\STATE{\textbf{Step 3}} If a termination criteria is met, stop and go to Step 4; else, repeat Step 2.
\STATE{\textbf{Step 4}} Compute $\bm \lambda^{k+1}$ using (\ref{eq:lambda}).
\STATE{\textbf{Step 5}} If a termination criteria is met, stop; else, set $k:=k+1$ and go to Step 1.
\end{algorithmic}
\end{algorithm}

The complexity analysis of Algorithm \ref{alg:ADMM-BCD} is presented as follows. First notice that Step 1 and Step 4 in Algorithm \ref{alg:ADMM-BCD} are identical to those in Algorithm \ref{alg:ADMM-LAP},  thus these two steps cost $O\left(n^{2}\right)$.

For Step 2 in Algorithm \ref{alg:ADMM-BCD}, we first analyze the computational cost of (\ref{xsubp}). To compute $\delta \bm x_{\mathcal{I}}$, we use the conjugate gradient method and set the maximum iteration number to a small constant. The main computational cost is from computing the multiplication of $\mu{\hat{\mathbf{J}}}_{\bm x}^{\top} {\hat{\mathbf{J}}}_{\bm x} + \nabla^2_{\bm x} {\hat R}\left(\bm x, \bm y^{k+1}, \bm \lambda^k\right)$ and a vector, and computing $\hat{{\mathbf{J}}}_{\bm x}^{\top} \bm r(\bm x^{(l)},\bm w^{(l)})$. Similar to the complexity analysis in Algorithm \ref{alg:ADMM-LAP}, this matrix-vector product can be done in $O\left(n^2 \log n\right)$ by FFTs. Then $\delta \bm x_{\mathcal{A}}$, $\delta \bm x$ can be computed in a similar way as discussed in Algorithm \ref{alg:ADMM-LAP} and a projected Armijo line search is applied to find the solution. The cost of these steps is $O\left(n^2 \log n\right)$. Hence, the total computational cost of (\ref{xsubp}) is {$O\left(n^{2}\log n\right)$}. As for the computational cost of (\ref{wsubp}), the main computational costs of computing $\delta \bm w_{\mathcal{I}}$ are from computing $\hat{\mathbf{J}}_{\bm w}^{\top}\hat{\mathbf{J}}_{\bm w}$ and ${\hat{\mathbf{J}}_{\bm w}^{\top} \bm r(\bm x^{(l+1)},\bm w^{(l)})}$. We know from the definition (\ref{jw}) that $\hat{\mathbf{J}}_{\bm w}^{\top}$ is a $p\times n^2$ matrix, $p\ll n^2$, hence the cost of computing $\hat{\mathbf{J}}_{\bm w}^{\top}\hat{\mathbf{J}}_{\bm w}$ is $O\left(n^{2}\right)$. The matrix-vector multiplication $\hat{\mathbf{J}}_{\bm w}^{\top}\bm r(\bm x^{(l+1)},\bm w^{(l)})$ can be done in $O\left(n^2 \log n\right)$ by FFTs. Moreover, the cost of computing $\delta \bm w_{\mathcal{A}}$, $\delta \bm w$ and applying the projected Armijo line search is $O\left(n^2 \log n\right)$. Therefore, the cost of (\ref{wsubp}) is $O\left(n^{2}\log n\right)$. As a result, the total computational cost of  Algorithm \ref{alg:ADMM-BCD} is {$O\left(n^{2}\log n\right)$.}

As shown from numerical experiments from Section \ref{sec:num}, Algorithm \ref{alg:ADMM-BCD} takes longer time than that of Algorithm \ref{alg:ADMM-LAP}. This is mainly because BCD requires computing the matrix-vector product $\hat{\mathbf{J}}_{\bm x}^{\top}\bm r(\bm x^{(l)},\bm w^{(l)})$ and $\hat{\mathbf{J}}_{\bm w}^{\top}\bm r(\bm x^{(l+1)},\bm w^{(l)})$ in each iteration, while LAP only requires computing $\hat{\mathbf{J}}_{\bm w}^{\top}\bm r(\bm x^{(l)},\bm w^{(l)})$. Hence, BCD takes more computational costs to compute the current residual value $\bm r$ and the matrix-vector multiplication $\hat{\mathbf{J}}_{\bm w}^{\top}\bm r$, which results in a larger number of FFTs. Moreover, the line search is applied twice in BCD to solve the search directions for both $\bm x$ and $\bm w$, while in LAP, only one line search is enough to obtain the search directions for both variables. Hence, Algorithm \ref{alg:ADMM-BCD} costs more. Moreover, it is important to point out that Algorithm \ref{alg:ADMM-LAP} takes the structure of the tightly coupled problem into consideration and converges faster.

\section{Numerical experiments}
\label{sec:num}
In this section, we illustrate the numerical performance of the proposed ADMM-LAP method for the myopic deconvolution problems with TV regularization arising from the AO retinal image restoration. All our computational results are obtained by MATLAB R2018b running on a Macbook Air with Intel Core i5 CPU (1.4 GHz).

The following notations will be used throughout the section:
\begin{itemize}
\item BlurLevel: an indicator used to set the severity of the blur to one of the following: `mild', `medium' and `severe';
\item $\text{NoiseLevel} \triangleq \|\bm{e}\|_{2} /\|\bm{d}^*\|_{2}$: relative level of noise, where $\bm e$ denotes the noise vector of perturbations and $\bm{d}^*$ denotes the exact (noise free) data vector;
\item $\# \rm iter$: the number of iterations performed by the algorithm;
\item ${\rm Rel. \ Err.} \ {\bm x}$: the relative error $\left\|\bm x-\bm x^{*}\right\|_{2} /\left\|\bm x^{*}\right\|_{2}$, where $\bm x^{*}$ and $\bm x$ denote the true image and computed image, respectively;
\item ${\rm Rel. \ Err.} \ {\bm w}$: the relative error $\left\| \bm w- \bm w^{*}\right\|_{2}/\left\| \bm w^{*}\right\|_{2}$, where $\bm w^{*}$ and $\bm w$ denote the true parameters and computed parameters, respectively;
\item $\operatorname{SNR}\left(\bm x\right) \triangleq 10 \cdot \log _{10} \frac{\|\overline{\bm x}-\tilde{\bm x}\|^{2}_{2}}{\|\overline{\bm x}-\bm x\|^{2}_{2}}$: signal-to-noise ratio (SNR), where $\overline{\bm x}$ is the original image, $\tilde{\bm x}$ is the mean intensity value of $\overline{\bm x}$ and $\bm x$ denotes the restored image. $\operatorname{SNR}\left(\bm x\right)$ measures the quality of restoration.
\end{itemize}

In all test problems, we set the regularization parameter $\mu=5\cdot 10^{4}$ by trial and errors such that the restored images had reasonable signal-to-noise ratio (SNR) and relative errors. The parameter $\xi$ was set to $\xi = 100$ such that the obtained parameters satisfied $\sum_{j = 1}^p w_j= 1$. The parameter $\beta$ was chosen according to the assumptions in Lemma \ref{lemma2}. In order to guarantee the sequence $\lbrace\epsilon_{k+1}\rbrace_{k=0}^{\infty}\subseteq[0,+\infty)$ satisfies $\sum_{k=0}^{\infty}\epsilon_{k+1}<\infty$, we choose $\epsilon_{k+1}=\frac{1}{a(k+1)^{2}}$, where $a$ is a constant. We used one uniform stopping criterion, that is
\begin{equation*}
\frac{|\hat{\Phi}(\bm x^{k+1}, \bm w^{k+1}, \bm y^{k+1},\bm \lambda^{k+1})-\hat{\Phi}(\bm x^{k}, \bm w^{k}, \bm y^{k},\bm \lambda^{k})|}{|\hat{\Phi}(\bm x^{k}, \bm w^{k}, \bm y^{k},\bm \lambda^{k})|}<\epsilon,
\end{equation*}
where $\hat{\Phi}(\bm x^{k}, \bm w^{k}, \bm y^{k},\bm \lambda^{k})$ is the objective function value in the $k$-th iteration and
$ \epsilon>0$ is a given tolerance.
In this paper, we set $\epsilon = 10^{-2}$. The maximum number of ADMM iterations was set to 50.
We also report the numerical results obtained by running the ADMM-BCD method. We compare the relative error of the image $\bm x$ and the relative error of the parameter $\bm w$ estimated by both ADMM-LAP and ADMM-BCD as well as their computational time and SNR of the restored images.

\begin{example}
\label{example1}
\rm
Adaptive optics (AO) flood illumination retinal imaging is a popular technique which has been in use for more than a decade\cite{Blanco2011, Blanco2014}. AO retinal imaging technique gives us the opportunity to observe and study retinal structures at the cellular level, which is impossible to see directly in the living eye. Due to the poor contrast of the raw AO retinal images, interpretation of such images requires the myopic deconvolution. For this example, we consider the case $p=2$. The global PSF used in the simulation is the sum of  two PSFs with the first one being focused and the second one being defocused. Let $A_1$, $A_2$ denote the relative blurring marices.
The test problem is a simulation generated from a $256\times 256$ pixel portion of an actual
AO retinal image.

We use the regularization toolbox IR Tools \cite{Gazzola2019} to build up the test problem.  \emph{PRblurgauss} and \emph{PRblurdefocus} are functions  used to simulate spatially invariant Gaussian blur and spatially invariant out-of-focus blurs, respectively.
Specifically,
\emph{PRblurgauss} constructs a $256 \times 256$ PSF array $P$ with entries
$$
    p_{ij} = \frac{1}{2\pi\sigma} \exp \left( - \frac{1}{2}
    \frac{(i-k)^2}{\sigma^2} +  \frac{(j-\ell)^2}{\sigma^2}\right)
$$
for a given value of $\sigma$, and
\emph{PRblurdefocus} constructs a $256 \times 256$ PSF array $P$ with entries
$$
    p_{ij} = \left\{ \begin{array}{ll}
      1/(\pi r^2), &  \mathrm{if} \ (i-k)^2+(j-\ell)^2 \leq r^2, \\[1mm]
      0, & \mathrm{elsewhere} \end{array} \right.
$$
for a given value of $r$.

In this example, three cases are considered. We fix one of the PSFs to be built up by \emph{PRblurgauss} with `mild' BlurLevel
($\sigma = 2$)
and set the other as a combined PSF built up by \emph{PRblurgauss} with `mild' BlurLevel and \emph{PRblurdefocus} with three different BlurLevels
(i.e., `mild' ($r = 7$), `medium' ($r = 15$), `severe' ($r = 31$)).
In addition, function \emph{PRnoise} is used to add Gaussian noise with NoiseLevel = 0.01. We set the true parameter $\bm w^{\ast}= [0.3; 0.7]$, choose the random image with the same size of the true image as the initial guess $\bm x^{0}$ and the initial guess $\bm {w}^0=[0.5; 0.5]$.

We then test ADMM-LAP and ADMM-BCD for the 2-weight case of (\ref{eq:objFctn1}) on this image.
For the case with the BlurLevel of \emph{PRblurdefocus} being set to be `medium', the restored images by both methods are shown in Figure {\ref{fig:3}}. We can clearly see that the restored image obtained by ADMM-LAP is much sharper, hence has a much better contrast than the one obtained by ADMM-BCD. Moreover, the restored image is much closer to the true image.

\begin{figure}[H]
\centering
\includegraphics[scale=0.5]{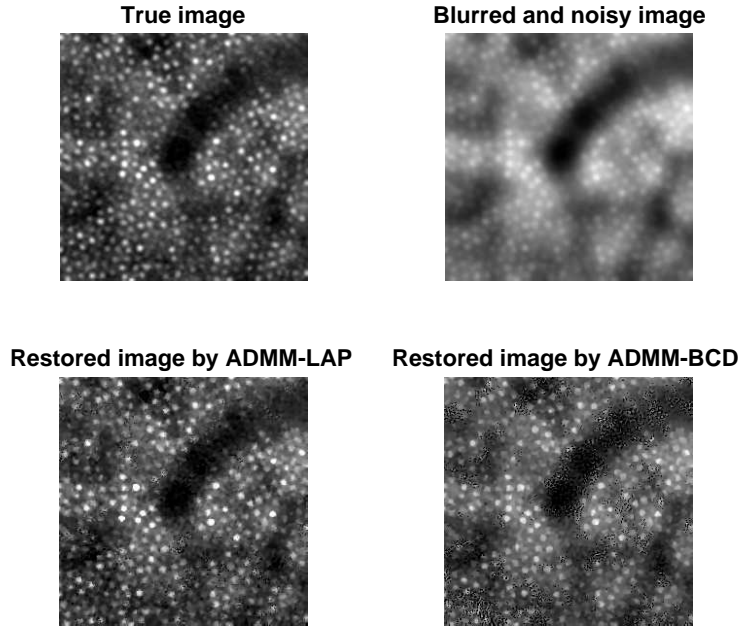}
\caption{Restored images for the retinal image in Example \ref{example1}. The BlurLevel of \emph{PRblurdefocus} is set to be `medium'. The images from the top left to the bottom right show the true image, the blurred and noisy image, the restored image by ADMM-LAP and the restored image by ADMM-BCD.}
\label{fig:3}
\end{figure}

The plots of the relative errors of the restored image $\bm x$ and the estimated parameters $\bm w$ against iteration can be seen in Figure \ref{fig:2}. It is easy to see that ADMM-LAP can reach lower relative errors than ADMM-BCD for both the restored image and the obtained parameters in this test.

\begin{figure}[H]
\centering
	\begin{subfigure}
		\centering
		\includegraphics[width=2.5in]{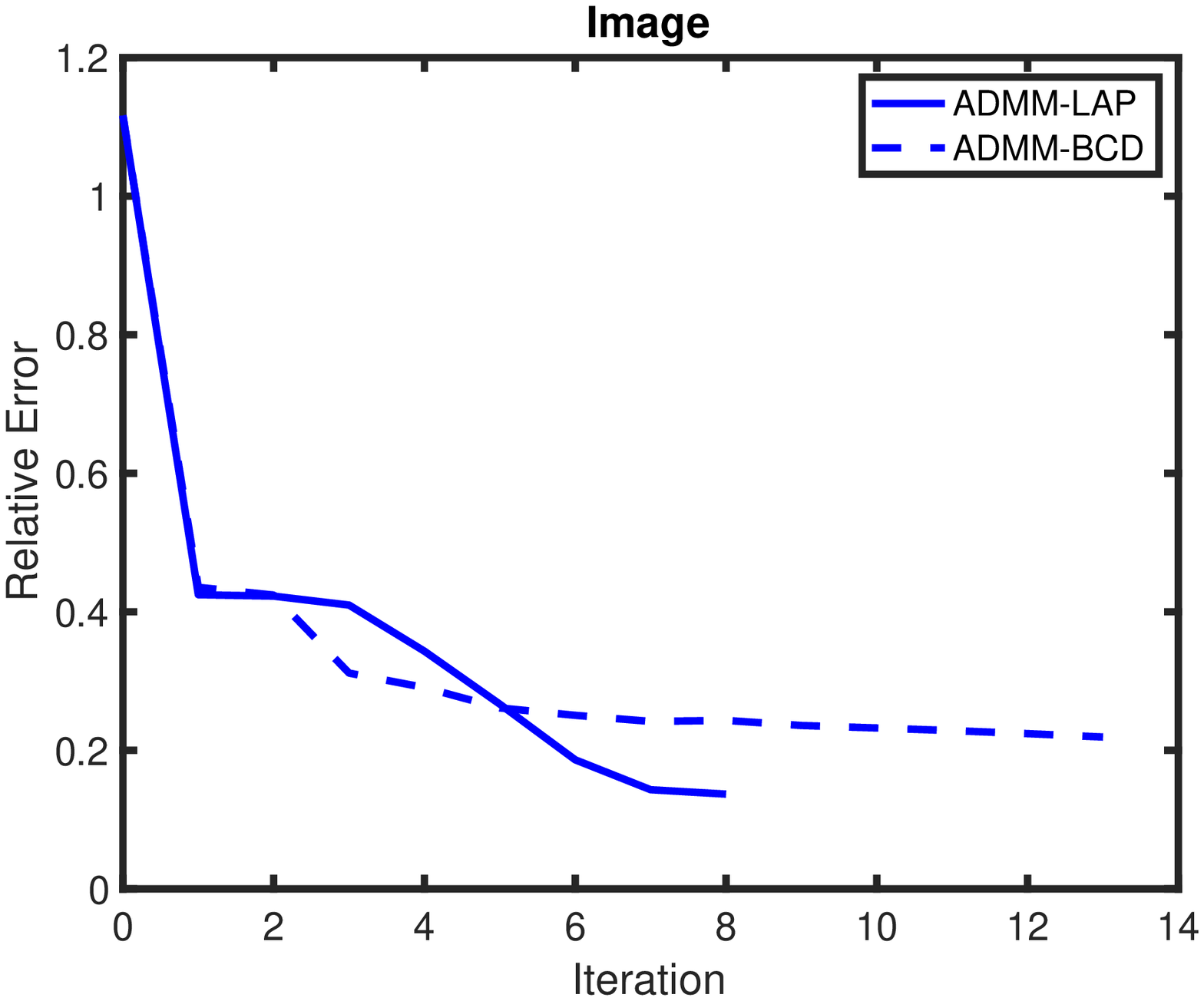}
	\end{subfigure}
	\qquad
	\begin{subfigure}
		\centering
		\includegraphics[width=2.5in]{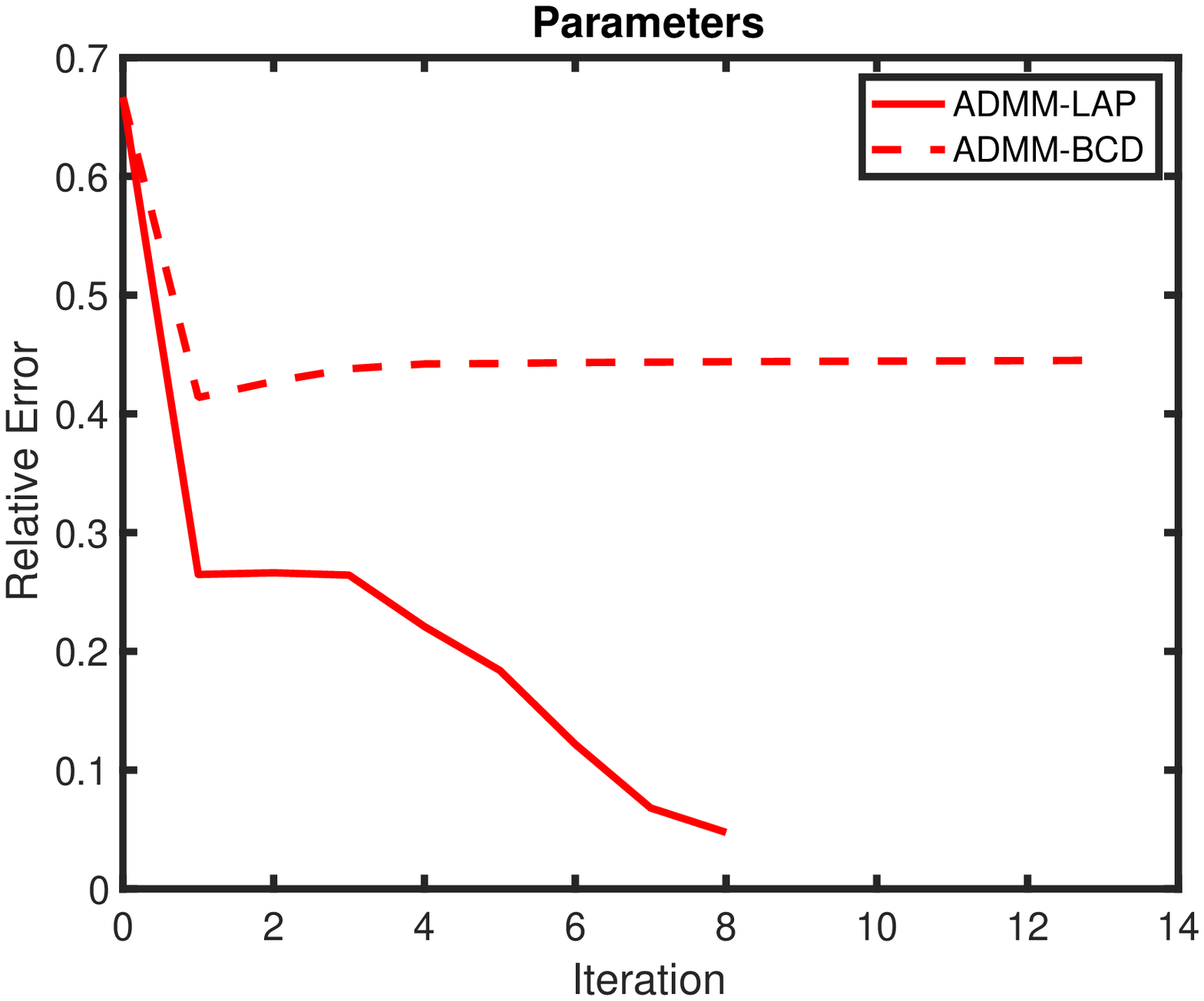}
	\end{subfigure}
	\caption{Relative errors of the restored image $\bm x$ (left) and the parameters $\bm w$ (right) for the myopic deconvolution problem in Example \ref{example1} with `medium' BlurLevel of \emph{PRblurdefocus}. }
	\label{fig:2}
\end{figure}

In Table \ref{tab:1}, we report the relative error of the restored image $\bm x$, the relative error of the parameters $\bm w$, the computational time and SNR of the restored images for both methods for all three BlurLevels. As can be seen from the fourth and fifth columns of Table \ref{tab:1}, ADMM-LAP can obtain more accurate restored images and more accurate parameters than ADMM-BCD on all three cases. The sixth column shows that ADMM-LAP is faster than ADMM-BCD, this is mainly because BCD takes a larger number of FFTs to compute the current residual value $\bm r$ and the matrix-vector multiplication $\hat{\mathbf{J}}_{\bm w}^{\top}\bm r$, and the line search is applied twice to solve the search directions for $\bm x$ and $\bm w$. Moreover, we can see from the seventh column that the quality of restored images obtained by ADMM-LAP is much better than those obtained by ADMM-BCD.
We point out that it is very difficult to analyze the difference of the relative errors of $\bm x$ and $\bm w$ between the test problems, especially because of the nonlinear relationship between $\bm x$ and $\bm w$.  We note that even for a linear problem, the quality of the error in $\bm x$ will depend not only on how ill-conditioned the matrix is, but on the distribution of singular values (e.g., is there a well-defined gap between large and small singular values, or do they decay smoothly). Our main purpose for this table of results is to compare the two methods (ADMM-LAP and ADMM-BCD) for each test problem, and to see that ADMM-LAP consistently outperforms ADMM-BCD even if the test problems change.

\begin{table}[H]
\centering
\caption{Numerical results of ADMM-LAP and ADMM-BCD for the AO retinal image in Example \ref{example1}. The columns from left to right give the BlurLevel, the method name,  the stopping iteration, the relative error of the restored image, the relative error of the solution parameters, the computational time in seconds and SNR of the restored images.}
\label{tab:1}
\begin{tabular}{ccccccc}
\noalign{\smallskip}\hline\noalign{\smallskip}
BlurLevel & Method &  $\#\rm iter$  & ${\rm Rel. \ Err.} \ {\bm x}$ & ${\rm Rel. \ Err.} \ {\bm w}$ & Time/s & SNR($\bm x$) \\
\noalign{\smallskip}\hline\noalign{\smallskip}
`mild'&ADMM-LAP&   6 &  \textbf{1.48e-01} & \textbf{2.63e-02} &\textbf{50.23} &\textbf{9.57} \\
\noalign{\smallskip}\cline{2-7}\noalign{\smallskip}
&ADMM-BCD  &  11 & 1.88e-01 & 1.11e-01 & 154.53 & 7.48\\
\noalign{\smallskip}\hline\noalign{\smallskip}
`medium'&ADMM-LAP&  8 & \textbf{1.37e-01} & \textbf{4.76e-02} &\textbf{56.42} & \textbf{10.26}\\
\noalign{\smallskip}\cline{2-7}\noalign{\smallskip}
&ADMM-BCD  &  13 & 2.19e-01 & 4.45e-01 & 145.09 &6.16\\
\noalign{\smallskip}\hline\noalign{\smallskip}
`severe'&ADMM-LAP&   12 &  \textbf{1.17e-01} & \textbf{4.39e-02} &\textbf{69.10} & \textbf{11.60} \\
\noalign{\smallskip}\cline{2-7}\noalign{\smallskip}
&ADMM-BCD  &  12 & 3.04e-01 & 9.13e-01 & 122.22 & 3.32\\
\noalign{\smallskip}\hline\noalign{\smallskip}
\end{tabular}
\end{table}
\end{example}

\begin{example}
\label{example2}
\rm
For this example, we consider the test problem from another simulation using a different $256\times 256$ portion of an actual AO retinal image with two blurring matrices. The size of the image is $256\times 256.$ We use \emph{PRblurgauss} with `mild' BlurLevel, \emph{PRblurdefocus} with three different BlurLevels (i.e., `mild', `medium', `severe') in IR Tools to build up a combined PSF, the other PSF is built up by \emph{PRblurgauss} with `mild' BlurLevel only. In addition, function \emph{PRnoise} is used to add Gaussian noise with NoiseLevel = 0.01. We set the true parameter $\bm w^{\ast}= [0.3; 0.7]$, choose the random image with the same size of the true image as the initial guess $\bm x^{0}$ and set the initial guess $\bm w^{0} = [w_{1}^{0}; w_{2}^{0}]\in \mathbb{R}^{2}$, where $w_{1}^{0}, w_{2}^{0}$ are random constants from $\left[0,1\right]$ satisfying $w_{1}^{0}+w_{2}^{0}=1$.

For the case with the BlurLevel of \emph{PRblurdefocus} being set to be `medium', the restored images by both ADMM-LAP and ADMM-BCD method are shown in Figure {\ref{fig:7}}. Like Example \ref{example1}, it is easy to see that the restored image obtained by ADMM-LAP is much sharper, and is much closer to the true image than the one obtained by ADMM-BCD.

\begin{figure}[H]
\centering
\includegraphics[scale=0.5]{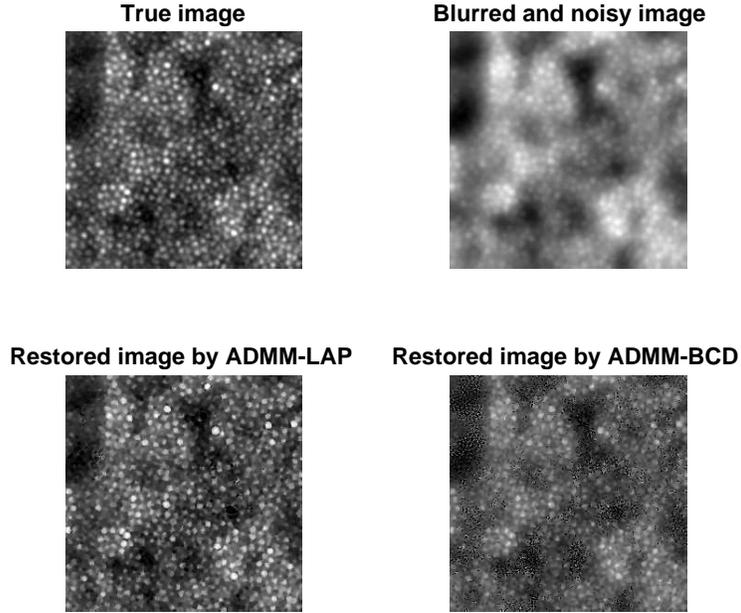}
\caption{Restored images for the retinal image in Example \ref{example2}. The BlurLevel of \emph{PRblurdefocus} is set to be `medium'. The images from the top left to the bottom right show the true image, the blurred and noisy image, the restored image by ADMM-LAP and the restored image by ADMM-BCD.}
\label{fig:7}
\end{figure}

We plot the relative errors of the restored image $\bm x$ and the estimated parameters $\bm w$ against iteration in Figure \ref{fig:8}. We can also see that ADMM-LAP reaches lower relative errors of both the restored image and the obtained parameters, hence it tends to recover more accurate restored images and parameters.

\begin{figure}[H]
\centering
	\begin{subfigure}
		\centering
		\includegraphics[width=2.5in]{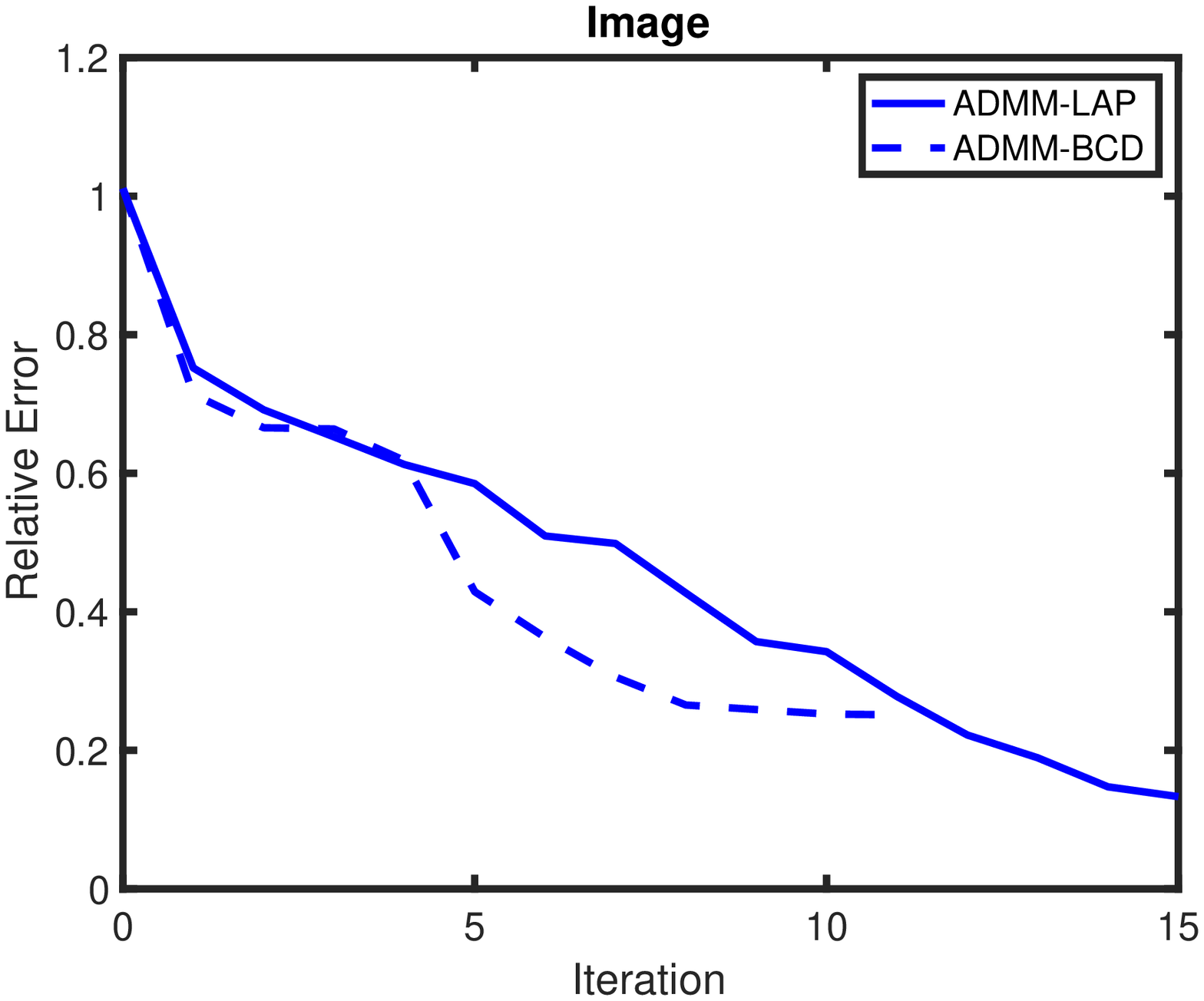}
	\end{subfigure}
	\qquad
	\begin{subfigure}
		\centering
		\includegraphics[width=2.5in]{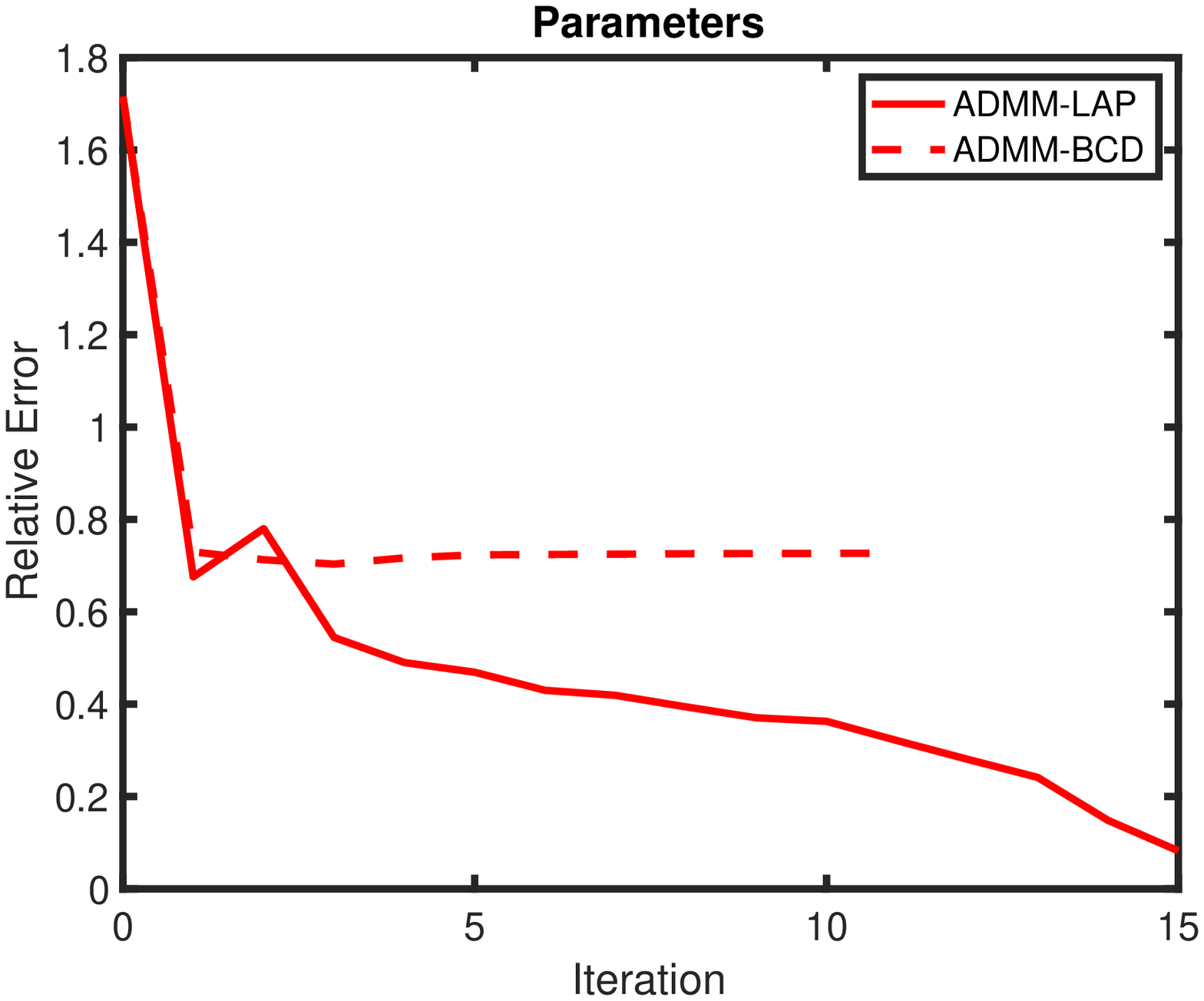}
	\end{subfigure}
	\caption{Relative errors for both the restored image $\bm x$ (left) and the parameters $\bm w$ (right) for the myopic deconvolution problem in Example \ref{example2} with `medium' BlurLevel of \emph{PRblurdefocus}.}
	\label{fig:8}
\end{figure}

Table \ref{tab:2} shows the relative error of the restored image $\bm x$, the relative error of the parameters $\bm w$, the computational time and SNR of the restored images associated with three cases for both methods. Similar to Example \ref{example1}, ADMM-LAP still outperforms ADMM-BCD in terms of the quality of the restored images and the accuracy of the obtained parameters by large margin. It can be clearly seen from the sixth column of Table \ref{tab:2} that ADMM-LAP is faster than ADMM-BCD. This is mainly because ADMM-BCD takes a larger number of FFTs and the line search is applied twice to compute the search directions, which is its computational bottleneck. The seventh column clearly shows that ADMM-LAP obtains restored images with a much better quality.
Above all, we can see that ADMM-LAP is much more efficient and accurate for this test example.

\begin{table}[H]
\centering
\caption{Numerical results of ADMM-LAP and ADMM-BCD for the retinal image in Example \ref{example2}. The columns from left to right give the BlurLevel, the method name,  the stopping iteration, the relative error of the restored image, the relative error of the solution parameters, the computational time in seconds and SNR of the restored images.}
\label{tab:2}
\begin{tabular}{cccccccc}
\noalign{\smallskip}\hline\noalign{\smallskip}
BlurLevel & Method&  $\#\rm iter$  & ${\rm Rel. \ Err.} \ {\bm x}$ & ${\rm Rel. \ Err.} \ {\bm w}$ & Time/s & SNR($\bm x$)\\
\noalign{\smallskip}\hline\noalign{\smallskip}
`mild'&ADMM-LAP& 11  &  \textbf{1.90e-01} & \textbf{1.01e-01} &\textbf{64.76} &\textbf{6.52}\\
\noalign{\smallskip}\cline{2-7}\noalign{\smallskip}
&ADMM-BCD  &  14  & 2.19e-01 & 2.74e-01 & 152.13 & 5.31\\
\noalign{\smallskip}\hline\noalign{\smallskip}
`medium'&ADMM-LAP&  15 &  \textbf{1.33e-01} & \textbf{8.23e-02} &\textbf{88.12} &\textbf{9.63}\\
\noalign{\smallskip}\cline{2-7}\noalign{\smallskip}
&ADMM-BCD  &  11 &  2.51e-01 & 7.27e-01& 119.20 &4.12\\
\noalign{\smallskip}\hline\noalign{\smallskip}
{`severe'}&ADMM-LAP&   5  &  \textbf{1.62e-01} & 7.77e-01 &\textbf{40.99} &\textbf{7.97}\\
\noalign{\smallskip}\cline{2-7}\noalign{\smallskip}
&ADMM-BCD  &  5  & 2.83e-01 & \textbf{3.76e-01} & 76.84 & 3.08 \\
\noalign{\smallskip}\hline\noalign{\smallskip}
\end{tabular}
\end{table}
\end{example}

\begin{example}
\label{example3}
\rm
In this example, we consider the case where three blurring matrices are provided, i.e., $p = 3$. The test problem is generated from the same AO retinal image used in Example \ref{example1}. The size of the image is $256\times 256.$ The first PSF is built up by \emph{PRblurgauss} with `mild' BlurLevel. The second PSF is a combined PSF built up by \emph{PRblurgauss} with `mild' BlurLevel and \emph{PRblurdefocus} with `medium' BlurLevel. The third PSF is a combined PSF built up by \emph{PRblurgauss} with `mild' BlurLevel and \emph{PRblurdefocus} with two different BlurLevels (i.e., `mild' and `severe' ). In addition, function \emph{PRnoise} is used to add Gaussian noise with NoiseLevel = 0.01. The true parameter vector $\bm w^{\ast}=[w^{\ast}_{1}; w^{\ast}_{2}; w^{\ast}_{3}]\in \mathbb{R}^{3}$ is set to be a random vector that satisfies $w^{\ast}_{i}\geq 0$ and $\sum_{i=1}^{3}w^{\ast}_{i}=1$.
We choose the random image with the same size of the true image as the initial guess $\bm x^{0}$ and set the initial guess $\bm w^0=[\frac{1}{3}; \frac{1}{3}; \frac{1}{3}]$.

The restored images by both ADMM-LAP and ADMM-BCD method are shown in Figure {\ref{fig:a}}. Like Example \ref{example1} and Example \ref{example2}, we can also see that ADMM-LAP reaches a much sharper image, it tends to recover an image that is much closer to the true image than the one obtained by ADMM-BCD.

\begin{figure}[H]
\centering
\includegraphics[scale=0.5]{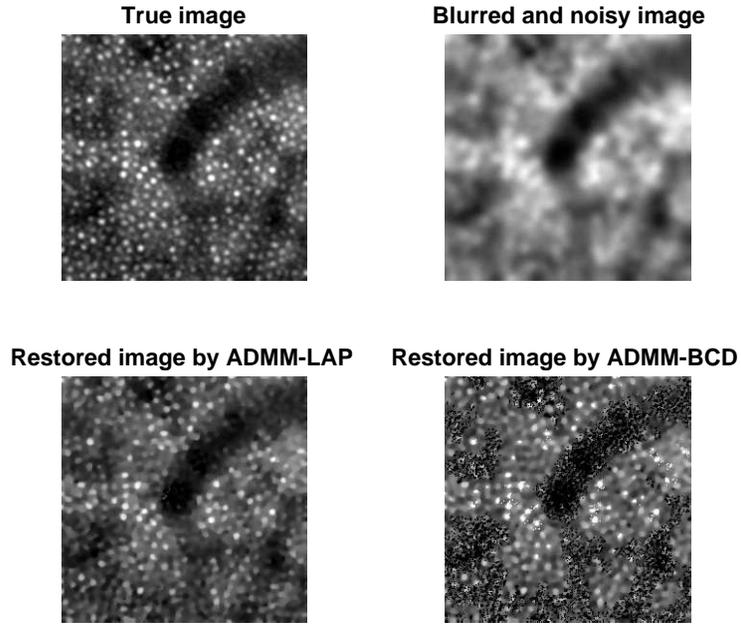}
\caption{Restored images for the retinal image in Example \ref{example3}. The BlurLevel of \emph{PRblurdefocus} in the third PSF is set to be `mild'. The images from the top left to the bottom right show the true image, the blurred and noisy image, the restored image by ADMM-LAP and the restored image by ADMM-BCD.}
\label{fig:a}
\end{figure}

We plot the relative errors of the restored image $\bm x$ and the estimated parameters $\bm w$ against iteration in Figure \ref{fig:b}. We can also see that ADMM-LAP achieve lower relative errors of both the restored image and the obtained parameters.

\begin{figure}[H]
\centering
	\begin{subfigure}
		\centering
		\includegraphics[width=2.5in]{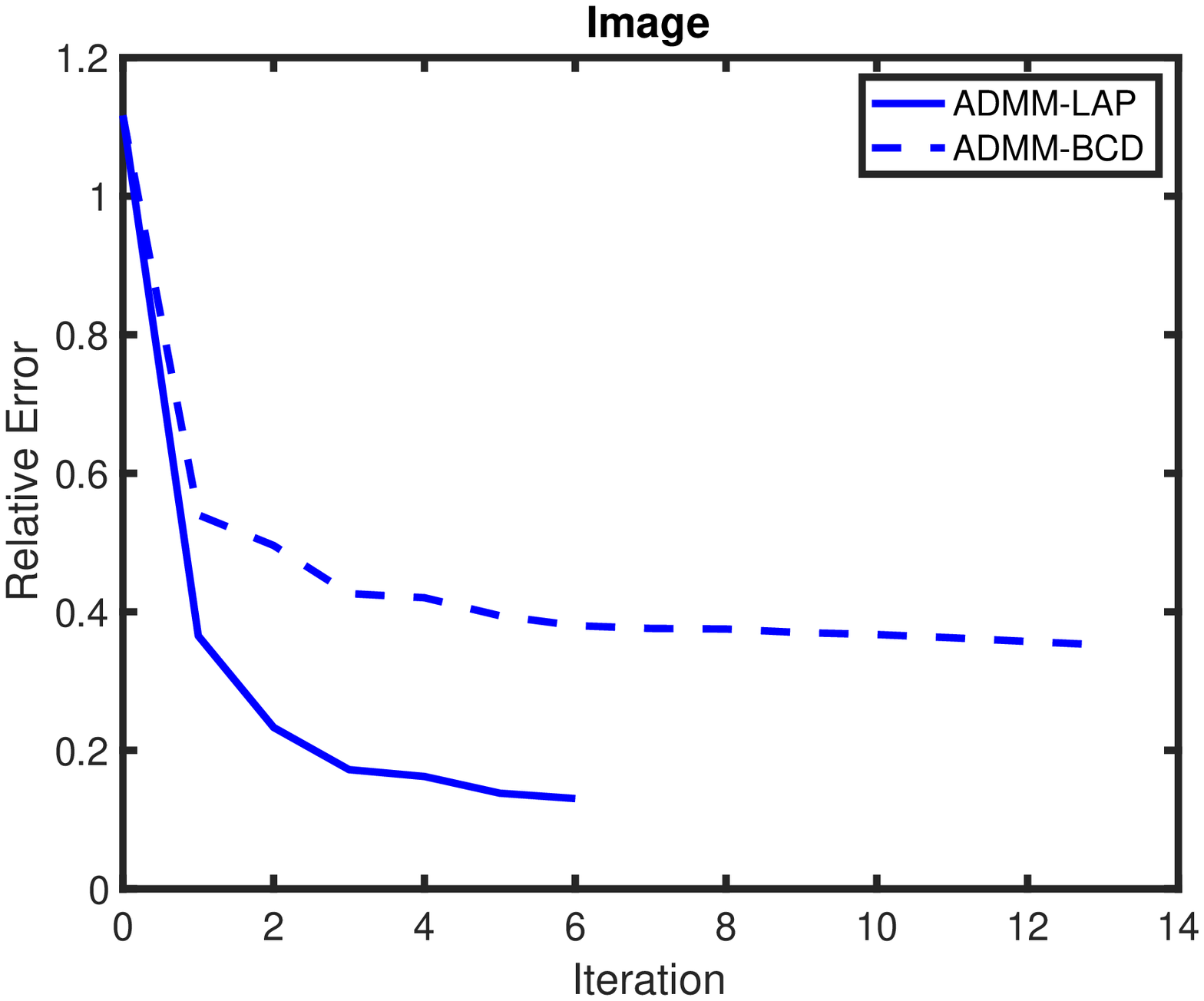}
	\end{subfigure}
	\qquad
	\begin{subfigure}
		\centering
		\includegraphics[width=2.5in]{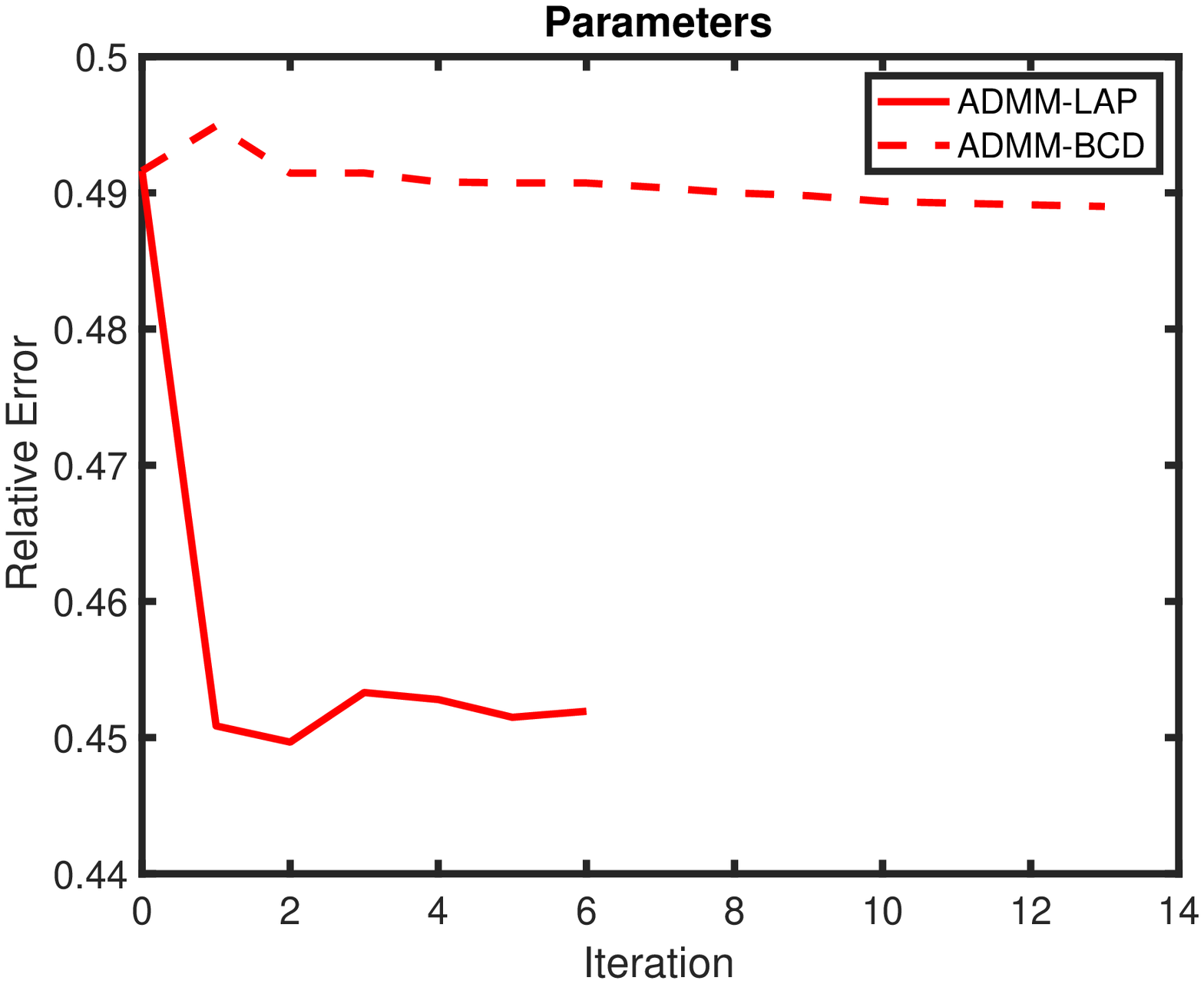}
	\end{subfigure}
	\caption{Relative errors of the restored image $\bm x$ (left) and the parameters $\bm w$ (right) for the myopic deconvolution problem in Example \ref{example3} with `mild' BlurLevel of \emph{PRblurdefocus} in the third PSF. }
	\label{fig:b}
\end{figure}

In Table \ref{tab:3}, we present the relative error of the restored image $\bm x$, the relative error of the parameters $\bm w$, the computational time and SNR of the restored images associated with two cases for both methods. We can also see from the sixth column that ADMM-LAP has evident advantage over ADMM-BCD in the computational time. Moreover, the seventh column shows that the quality of restored images obtained by ADMM-LAP is much better than those obtained by ADMM-BCD.

\begin{table}[H]
\centering
\caption{Numerical results of ADMM-LAP and ADMM-BCD for the retinal image in Example \ref{example3}. The columns from left to right give the BlurLevel, the method name,  the stopping iteration, the relative error of the restored image, the relative error of the solution parameters, the computational time in seconds and SNR of the restored images.}
\label{tab:3}
\begin{tabular}{ccccccc}
\noalign{\smallskip}\hline\noalign{\smallskip}
BlurLevel & Method&  $\#\rm iter$ &  ${\rm Rel. \ Err.} \ {\bm x}$ & ${\rm Rel. \ Err.} \ {\bm w}$ & Time/s & SNR($\bm x$)\\
\noalign{\smallskip}\hline\noalign{\smallskip}
`mild' &ADMM-LAP&  6 &  \textbf{1.30e-01} & \textbf{4.52e-01} &\textbf{134.03} &\textbf{10.68}\\
\noalign{\smallskip}\cline{2-7}\noalign{\smallskip}
 & ADMM-BCD  &  13 &  3.52e-01 & 4.89e-01& 267.21 &2.05\\
\noalign{\smallskip}\hline\noalign{\smallskip}
`severe' &ADMM-LAP&  11 &  \textbf{1.92e-01} & \textbf{1.79e-01} &\textbf{156.05} &\textbf{7.33}\\
\noalign{\smallskip}\cline{2-7}\noalign{\smallskip}
 & ADMM-BCD  &  10 &  2.18e-01 & 3.27e-01& 186.31 &6.19\\
\noalign{\smallskip}\hline\noalign{\smallskip}
\end{tabular}
\end{table}
\end{example}

\section{Conclusion}
\label{sec:5}
In this paper, we propose a new efficient ADMM-LAP method for solving large scale ill-posed inverse problems, and more specifically myopic deconvolution problems with TV regularization arising from the adaptive optics retinal image restoration. Specifically, ADMM is applied to tackle the nondifferentiable and nonlinear TV regularization term first, then LAP is applied to tackle tightly coupled $\left(\bm x,\bm w\right)$-subproblems appearing within each ADMM iteration.
The convergence results of ADMM-LAP are presented. Moreover, the efficiency of the proposed algorithm is demonstrated with a theoretical computational complexity analysis.
In our numerical experiments we show that the proposed ADMM-LAP method is superior to ADMM-BCD method in terms of both the accuracy and the efficiency. In our future work, we plan to exploit appropriate preconditioning techniques to further reduce the iteration number and the iteration time.

\section*{Acknowledgment}
JN acknowledges support from the U.S. National Science Foundation under grant DMS- 1819042 and the National Institutes of Health under grant 1R13EB028700-01. The authors also thank Dr. John M. Nickerson, Professor of Ophthalmology at Emory University for helpful discussions and for providing data used for numerical experiments in this paper.

\end{document}